\documentclass[11pt]{article}
\usepackage[utf8]{inputenc}
\usepackage{lmodern}
\usepackage{subfiles}
\usepackage{enumitem}
	\setenumerate{label={\normalfont(\roman*)}, itemsep=0em} 
        
\usepackage{amsfonts}
\usepackage{amsthm}
\usepackage{amsmath}
\usepackage{amssymb}
\usepackage{amscd}
\usepackage{mathrsfs}
\usepackage{mathtools}
\usepackage{bm}
\usepackage{esint}

\usepackage[margin=3cm]{geometry}
\usepackage{indentfirst}
\usepackage{graphicx}
\usepackage{graphics}
\usepackage{lscape}
\usepackage{tikz-cd}
\usepackage{color}
\usepackage{pict2e}
\usepackage{epic}
\usepackage{epstopdf}
\usepackage{titlesec, titlefoot}
	\titleformat{\section}[block]{\Large\bfseries\filcenter}{\thesection}{1em}{}
\usepackage{commath}
\usepackage{float}
\usepackage{caption}
\usepackage{etoolbox}
\usepackage[affil-it]{authblk}
\usepackage{combelow}

\usepackage[hidelinks]{hyperref}
\hypersetup{bookmarksopen=true} 
\usepackage{hypcap}

\graphicspath{{./Pictures/}}
\allowdisplaybreaks

\usepackage{listings}
\usepackage{fancybox}

\theoremstyle{plain}

\renewcommand*\thesection{\arabic{section}}
\numberwithin{equation}{section} 

\theoremstyle{plain}
\newtheorem{thm}{Theorem}

\newtheorem{lemma}[thm]{Lemma}

\numberwithin{thm}{section} 
\newtheorem{corollary}{Corollary}
\newtheorem{theorem}{Theorem}
\newtheorem{proposition}{Proposition}

\theoremstyle{definition}

\newtheorem{remark}[thm]{Remark}

\newtheorem{definition}[thm]{Definition}

\newcommand{\thistheoremname}{}
\newtheorem{genericthm}[equation]{\thistheoremname}

\newcommand{\thistheoremnames}{}
\newtheorem*{genericthms}{\thistheoremnames}
\newenvironment{para*}[1]
  {\renewcommand{\thistheoremnames}{#1}%
   \begin{genericthms}}
  {\end{genericthms}}

\expandafter\let\expandafter\oldproof\csname\string\proof\endcsname
\let\oldendproof\endproof
\renewenvironment{proof}[1][\proofname]{%
  \oldproof[\upshape \bfseries #1]%
}{\oldendproof}

\makeatletter
\def\@makechapterhead#1{%
  \vspace*{50\p@}%
  {\parindent \z@ \raggedright \normalfont
    \interlinepenalty\@M
    \Huge\bfseries  \thechapter.\quad #1\par\nobreak
    \vskip 40\p@
  }}
\makeatother

\newcommand{\reqnomode}{\tagsleft@false}
\def \d{\,{\rm d}}
\def\dist{\,{\rm dist}}
\def\supp{\,{\rm supp }}

\newcommand{\const}{\operatorname{const}}

\allowdisplaybreaks
\makeatletter
\DeclareRobustCommand*{\bfseries}{%
  \not@math@alphabet\bfseries\mathbf
  \fontseries\bfdefault\selectfont
  \boldmath
}

\makeatother

\newlength{\defbaselineskip}
\newcommand{\N}{\mathbb{N}}
\newcommand\eps\varepsilon
\def\mean#1{\mathchoice%
          {\mathop{\kern 0.2em\vrule width 0.6em height 0.69678ex depth -0.58065ex
                  \kern -0.8em \intop}\nolimits_{\kern -0.4em#1}}%
          {\mathop{\kern 0.1em\vrule width 0.5em height 0.69678ex depth -0.60387ex
                  \kern -0.6em \intop}\nolimits_{#1}}%
          {\mathop{\kern 0.1em\vrule width 0.5em height 0.69678ex
              depth -0.60387ex
                  \kern -0.6em \intop}\nolimits_{#1}}%
          {\mathop{\kern 0.1em\vrule width 0.5em height 0.69678ex depth -0.60387ex
                  \kern -0.6em \intop}\nolimits_{#1}}}

\numberwithin{equation}{section}
\allowdisplaybreaks[4]

\newcommand\blfootnote[1]{%
  \begingroup
  \renewcommand\thefootnote{}\footnote{#1}%
  \addtocounter{footnote}{-1}%
  \endgroup
}

\def\eqn#1$$#2$${\begin{equation}\label#1#2\end{equation}}
\newcommand\R{\mathbb{R}}

\newcommand{\F}{\mathscr F}

\def \tp{\textup}
\def \p{\partial}
\def \e{\varepsilon}

\newcommand\restr[2]{{
  \left.\kern-\nulldelimiterspace 
  #1 
  \vphantom{|} 
  \right|_{#2} 
  }}

\title{Global higher integrability for minimisers of convex functionals with (\MakeLowercase{p},\MakeLowercase{q})-growth}
\author{Lukas Koch}

\affil[1]{\small University of Oxford, Andrew Wiles Building Woodstock Rd, Oxford OX2 6GG, United Kingdom 
\protect \\
  {\tt{kochl@maths.ox.ac.uk}}
  \vspace{1em} \ }

\usepackage{etoolbox}
\makeatletter
\patchcmd{\@adjustvertspacing}
  {\jot\baselineskip \divide\jot 4}
  {\jot=.5\baselineskip}
  {}{}
\makeatother

\makeatother
\begin{document}
\maketitle
\begin{abstract}
We prove global $W^{1,q}(\Omega,\R^m)$-regularity for minimisers of convex functionals of the form $\F(u)=\int_\Omega F(x,Du)\d x$. $W^{1,q}(\Omega,\R^m)$ regularity is also proven for minimisers of the associated relaxed functional. Our main assumptions on $F(x,z)$ are a uniform $\alpha$-H\"older continuity assumption in $x$ and controlled $(p,q)$-growth conditions in $z$ with $q<\frac{(n+\alpha)p}{n}$.
\end{abstract}

\blfootnote{
\emph{Acknowledgements.} L.K. was supported by the Engineering and Physical Sciences Research Council [EP/L015811/1].
}

\section{Introduction and results}
We study minimisation problems of the form
\begin{align*}
\min_{u\in W^{1,p}_g(\Omega,\R^m)} \F(u) \qquad \text{ where }
\F(u)=\int_\Omega F(x,Du)- f\cdot u\d x.
\end{align*}
Here $F(x,z)$ is a convex functional with $(p,q)$-growth in $z$ satisfying a natural uniform $\alpha$-H\"older condition in $x$. Furthermore $\Omega$ is a sufficiently regular domain in $\R^n$, ${f\in L^{q'}(\Omega,\R^m)}$ and $g$ is a sufficiently regular boundary datum. We assume throughout that $1<p\leq q<\infty$ and will impose further restrictions as necessary. For precise definitions we refer to Section \ref{sec:prelim} where we also explain our notation. We prove global higher integrability properties of the minimiser as well as global higher integrability properties of minimisers of a relaxed functional related to $\F(\cdot)$. Our results concern mostly the vectorial case $n,m\geq 2$. The one-dimensional case $n=1$ is special and stronger results apply.

The study of elliptic systems and functionals in the case where $p=q$ is well established with a very long list of important results. For an introduction and references we refer to \cite{Giaquinta1983} and \cite{Giusti2003}. 

In order to state our results precisely and to compare them with the literature on functionals with $(p,q)$-growth we briefly state our assumptions precisely.

\textit{Let $0<\alpha\leq 1$. We assume that $F(x,z)$ is measurable in $x$, continuously differentiable in $z$ and moreover satisfies 
\begin{gather}\label{def:bounds1}
\nu(\mu^2+|z|^2+|w|^2)^\frac{p-2}{2}\leq\frac{F(x,z)-F(x,w)-\langle \partial_z F(x,w),z-w\rangle}{|z-w|^2}\tag{H1}\\
\label{def:bounds2}
|F(x,z)|\lesssim (1+|z|^2)^\frac q 2\tag{H2}\\
\label{def:bounds3} |F(x,z)-F(y,z)|\leq \Lambda |x-y|^\alpha\left(1+ |z|^2\right)^\frac q 2.\tag{H3}
\end{gather}
for some $\mu, \nu, \Lambda>0$, all $z,w\in \R^{n\times m}$ and almost every $x,y\in \Omega$, where $1\leq p\leq q$. If $p<n$ further suppose $q\leq \frac{np}{n-p}$.}

 Note that under these assumptions $F(x,z)$ is convex. We also remark that our conditions clearly imply the following two bounds for $z\in \R^{n\times m}$ and almost every $x\in\Omega$,
\begin{gather}
\label{eq:pnormlower}|z|^p-1\lesssim F(x,z)\\
\label{eq:upperBoundDerivF} |\p_z F(x,z)|\lesssim \Lambda \left(1+|z|^2\right)^\frac{q-1} 2.
\end{gather}

We refer to Section \ref{sec:examples} for examples of functionals satisfying these assumptions. Alternative assumptions to \eqref{def:bounds1} that allow for anisotropic or $p(x)$-growth and under  which the conclusions of Theorem \ref{thm:regularityRelaxed} and Theorem \ref{thm:nonautonReg} still hold are explored in Section \ref{sec:alternatives}.

Let us also give a precise definition of the notions of minimisers we are interested in:
\begin{definition}
$u\in W_g^{1,1}(\Omega)$ is a \textbf{(pointwise) minimiser} of $\F(\cdot)$ in the class $W^{1,p}_g(\Omega)$ if it holds that ${F(x,Du)\in L^1(\Omega)}$ and
\begin{align*}
\int_\Omega F(x,Du)-f\cdot u\d x \leq \int_{\Omega} F(x,Du+D\phi)- f\cdot(u+\phi)\d x
\end{align*}
for all $\phi \in W^{1,1}_0(\Omega)$. 

$u\in W^{1,p}_g(\Omega)$ is a \textbf{$W^{1,q}$-relaxed minimiser} (usually referred to as a relaxed minimiser) of $\F(\cdot)$ in the class $W^{1,p}_g(\Omega)$ if $u$ minimises the relaxed functional
\begin{align}\label{def:relaxedFunctional}
\overline \F(v) = \inf \left\{\,\liminf_{j\to\infty} \int_\Omega F(D v_j)-f\cdot v_j: (v_j)\subset Y, v_j\rightharpoonup v \text{ weakly in } X\,\right\}
\end{align}
amongst all $v\in X= W^{1,p}_g(\Omega)$ where $Y=W^{1,q}_g(\Omega)$. 
\end{definition} 
Note that for a pointwise minimiser, $\F(u)<\infty$, 
coupled with (\ref{def:bounds1}), Sobolev embedding and Young's inequality gives $Du\in L^p(\Omega)$. We remark further that by weak lower semicontinuity of $\F(\cdot)$ in $W^{1,q}(\Omega)$, for $u\in W^{1,q}_g(\Omega)$, $\overline \F(u)=\F(u)$. 

The study of regularity theory for minimisers in the case $p<q$ started with the seminal papers \cite{Marcellini1989}, \cite{Marcellini1991}. We don't aim to give a complete overview of the theory here, focusing on results directly relevant to this paper. We refer to \cite{Mingione2006} for a good overview and further references. A particular focus of research have been the special cases of the double-phase functional $F(x,z)=|z|^p+a(x)|z|^q$ and functionals with $p(x)$-growth. For an introduction and further references with regards to these special cases we refer to the introduction of \cite{Baroni2018} and \cite{Diening2011}, \cite{Radulescu2015}, respectively.  Already in the scalar $m=1$ autonomous case $F(x,z)=F(z)$ counterexamples show that in order to prove regularity of minimisers $p$ and $q$ may not be too far apart \cite{Giaquinta1987}, \cite{Marcellini1989}, \cite{Hong1992}. We stress that the counterexamples only apply to pointwise minimisers (as opposed to relaxed minimisers, see Definition \ref{def:relaxedFunctional}). A crucial first step to proving interior regularity results is to obtain $W^{1,q}_{\tp{loc}}$ regularity of minimisers. Once $W^{1,q}_{\tp{loc}}$ regularity is obtained, further regularity results such as partial regularity may be obtained, see for example \cite{Acerbi1994}, \cite{PassarellidiNapoli1996}. We list the to our knowledge best available $W^{1,q}_{\tp{loc}}$ regularity results for general autonomous convex functionals with $(p,q)$-growth (when $n\geq 2$): Under natural growth conditions it suffices to assume $q<\frac{np}{n-1}$ \cite{Carozza2013} in order to obtain $W^{1,q}_{\tp{loc}}$ regularity of minimisers. To obtain the same conclusion under controlled growth conditions the gap may be widened to $q<p\left(1+\frac 2 {n-1}\right)$ \cite{Schaffner2020} and under controlled duality growth conditions it suffices to take $q<\frac{np}{n-2}$ (if $n=2$ it suffices to take $q<\infty$) \cite{DeFilippis2020}. We note that in all three cases higher integrability goes hand in hand with a higher differentiability result. Partial $C^{1,\alpha}$ regularity of pointwise and relaxed minimisers (with $Y=W^{1,q}_{\tp{loc}}$) for autonomous quasiconvex functionals has been obtained in \cite{Schmidt2008,Schmidt2008a,Schmidt2009} under natural growth conditions with $q<1+\frac{\min (2,p)}{2n}$. The borderline-case $p=n-1$, $q=n$ was studied in \cite{Carozza2020}. Finally, we refer to \cite{Bogelein2013} for results and references in the case of parabolic systems with $(p,q)$-growth.

We now turn to the case of non-autonomous functionals $F(x,z)$, convex and with $(p,q)$-growth in $z$, while satisfying a uniform $\alpha$-H\"older condition in $x$. For $n\geq 2$ counterexamples to $W^{1,q}$ regularity with $1<p<n<n+\alpha<q$ are due to \cite{Esposito2004}, see also \cite{Fonseca2004}. Recent work suggests that the condition $p<n<q$ may be removed \cite{Diening2020}. If $q<\frac{(n+\alpha)p}{n}$, it was proven in \cite{Esposito2004} for many standard examples that minimisers enjoy $W^{1,q}_{\tp{loc}}$ regularity. Using \cite{Esposito2019} the result may be extended to functionals satisfying in addition
the following condition:

\textit{there is $\e_0>0$ such that for any $\e\in(0,\e_0)$ and $x\in\Omega$ there is $\hat y\in \overline{B_\e(x)\cap\Omega}$ such that 
\begin{align}\label{def:changeOfXAlt}
F(\hat y,z)\leq F(y,z) \qquad\forall y\in \overline{B_\e(x)\cap\Omega},\hspace{0.5cm} z\in \R^{n\times m}.\tag{H4}
\end{align}
}
\eqref{def:changeOfXAlt} is very similar to Assumption 2.3. in \cite{Zhikov1995}. 
We remark that \eqref{def:changeOfXAlt} holds for many of the commonly considered examples, see Section \ref{sec:examples} as well as \cite{Esposito2004}, \cite{Esposito2019}. $W^{1,q}_{\tp{loc}}$ regularity is in general not known if $q=\frac{(n+\alpha)p}{n}$. An exception are functionals modeled on the double-phase functional \cite{Baroni2018}, see also \cite{DeFilippis2019}.

One of the few results on regularity up to the boundary in the setting of $(p,q)$-growth functionals is \cite{Bulicek2018} where Lipschitz regularity up to the boundary is obtained for minimisers of scalar autonomous functionals satisfying nonstandard growth conditions and the structure condition $F(x,z)=b(|z|)$. The growth conditions considered include $(p,q)$-growth. For functionals satisfying a structural assumption inspired by the double-phase functional Cald\'{e}ron-Zygmund estimates valid up to the boundary are obtained in \cite{Byun2017}. H\"older-regularity up to the boundary for double-phase functionals is studied in \cite{Tachikawa2020}.

If additional structure assumptions such as $F(x,z)=b(x,|z|)$ are imposed or if it is assumed that minimisers are bounded it is possible to improve on the results listed so far. Without going into further detail we refer to \cite{Breit2012}, \cite{Carozza2011} for results and further references in these directions. Local boundedness of minimisers for convex non-autonomous $(p,q)$-growth functionals under natural growth conditions and the additional assumption $F(x,2z)\lesssim 1+F(x,z)$ is studied in \cite{Hirsch2020}.

The main result of this paper is to obtain global $W^{1,q}$ regularity of minimisers for non-autonomous functionals $F(x,z)$- convex and with controlled $(p,q)$-growth conditions in $z$ as well as satisfying a uniform $\alpha$-H\"older condition in $x$- if $q<\frac{(n+\alpha)p}{n}$ and under the additional assumption \eqref{def:changeOfXAlt}. Thus our results are the global analogue of the local results in \cite{Esposito2004}, \cite{Esposito2019}. We stress that we do not make any structure assumption on the dependence on $z$ going beyond controlled $(p,q)$-growth conditions and that we make no geometric assumptions on either the domain or the boundary datum $g$.
To the best of our knowledge this is the first global $W^{1,q}$ regularity result valid for a large class of general convex $(p,q)$-growth functionals. 

Proofs of $W^{1,q}_{\tp{loc}}$-regularity in the $(p,q)$-growth framework often involve two major steps. First, an a-priori estimate for minimisers of a suitably regularised $q$-growth functional is obtained using difference quotient methods. Second, taking limits the estimate is transferred to minimisers of the original functional. In this paper we follow the same ideas. In order to apply the difference quotient method globally, we rely on an argument developed in \cite{Savare1998}. Using this technique we obtain a regularity result for relaxed minimisers:
\begin{theorem}\label{thm:regularityRelaxed}
Suppose $\Omega$ is a Lipschitz domain.
Let $g\in W^{1+\alpha,q}(\Omega)$.
Suppose $F(x,\cdot)$ satisfies \eqref{def:bounds1}-\eqref{def:bounds3} with $1< p\leq q< \frac{(n+\alpha)p}{n}$. If $u$ is a relaxed minimiser of $\F(\cdot)$ in the class $W^{1,p}_g(\Omega)$, then $u\in W^{1,q}(\Omega)$. Moreover for any $0\leq\beta < \alpha$ there is $\gamma>0$ such that
\begin{align*}
\|u\|_{W^{1,\frac{np}{n-\beta}}(\Omega)} \lesssim \left(1+\overline\F(u)+\|g\|_{W^{1+\alpha,q}(\Omega)}+\|f\|_{L^{q'}(\Omega)}\right)^\gamma.
\end{align*}
\end{theorem}

A key obstruction to $W^{1,q}$-regularity in the $(p,q)$-growth setting is the Lavrentiev phenomenon, which describes the possibility that
\begin{align}\label{eq:Lphenomenon}
\inf_{u\in W^{1,p}_g(\Omega)} \F(u)< \inf_{u\in W^{1,q}_g(\Omega)} \F(u).
\end{align}
A first example of this phenomenon was given in \cite{Lavrentiev1926}. In the context of $(p,q)$- growth functionals the theory was further developed in \cite{Zhikov1987}, \cite{Zhikov1993},\cite{Zhikov1995}. The Lavrentiev phenomenon is closely related to properties of the relaxed functional. We adopt the viewpoint and terminology of \cite{Buttazo1992} and consider a topological space $X$ of weakly differentiable functions and a dense subspace $Y\subset X$. We introduce the following sequentially lower semi-continuous (slsc) envelopes
\begin{align*}
\overline \F_X = \sup\{\,\mathscr G\colon X\to[0,\infty]: \mathscr G \text{ slsc }, \mathscr G\leq \F \text{ on } X\,\}\\
\overline \F_{Y} = \sup\{\,\mathscr G\colon X\to[0,\infty]: \mathscr G \text{ slsc }, \mathscr G\leq \F \text{ on } Y\,\}\nonumber
\end{align*}
and define the Lavrentiev gap functional for $u\in X$ as
\begin{align*}
\mathscr L(u,X,Y)=\begin{cases}
		\overline \F_{Y}(u)-\overline \F_X(u)  &\text{ if } \overline \F_X(u)<\infty\\
		0 &\text{ else}.
	\end{cases}
\end{align*}
Note that the gap functional is non-negative. 

There is an extensive literature on the Lavrientiev phenomenon and gap functional, an overview of which can be found in \cite{Buttazo1995}, \cite{Foss2001} to which we also refer for further references. The phenomenon is also of interest in nonlinear elasticity \cite{Foss2003}. Considering the common choice $X=W^{1,p}(\Omega)$ endowed with the weak toplogy, $Y=W^{1,q}_{\tp{loc}}(\Omega)$ a question related to the Lavrentiev phenomenon is to study measure representations of $\overline \F(\cdot)$. We refer to \cite{Fonseca1997,Acerbi2003} for results and further references in this direction, but point out that in this context in \cite{Fonseca2005a} an argument using a Whitney cover of $\Omega$ was given that is similar to arguments in this paper.

In this paper we always consider the choice $X=W^{1,p}_g(\Omega)$ endowed with the weak topology and $Y=W^{1,q}_g(\Omega)$.
  Since $F(x,z)$ is convex, standard methods show that then $\overline\F_X(\cdot)=\F(\cdot)$, see \cite[Chapter 4]{Giusti2003}. Further $\overline \F_Y(\cdot)=\overline \F(\cdot)$. We also note that if ${\mathscr L(u,X,Y)=0}$ for all $u\in X$, then the Lavrentiev phenomenon cannot occur. Non-occurrence of the Lavrentiev phenomenon allows to transfer the estimates obtained in Theorem \ref{thm:regularityRelaxed} to pointwise minimisers and thus to establish $W^{1,q}$-regularity.
 In fact under the assumption that $\mathscr L(u,X,Y)=0$ for all $u\in X$ with $X=W^{1,p}(\Omega)$, endowed with the weak toplogy, and $Y=W^{1,p}(B_r)$ with $B_r\Subset\Omega$, $W^{1,q}(B_r)$-regularity of minimisers of non-autonomous functionals satisfying \eqref{def:bounds1}-\eqref{def:bounds3} with $1<p\leq q<\frac{(n+\alpha)p}{n}$ is obtained in  \cite{Esposito2004}. Under \eqref{def:bounds1}-\eqref{def:bounds3} and the same restriction on $q$, \cite{Esposito2019} shows that $\mathscr L(\cdot,W^{1,p}(\Omega),W^{1,q}(B_r))=0$ where $B_r\Subset\Omega$ if in addition \eqref{def:changeOfXAlt} holds. Our next theorem obtains global $W^{1,q}$-regularity of minimisers under the same assumptions.
\begin{theorem}\label{thm:nonautonReg}
Suppose $\Omega$ is a $C^{1,\alpha}$-domain.
Suppose $g\in W^{1+\max\left(\alpha,\frac 1 q\right),q}(\Omega)$. Assume $\F(\cdot)$ satisfies \eqref{def:bounds1}-\eqref{def:bounds3} with $1< p\leq q<\frac{(n+\alpha)p}{n}$ and \eqref{def:changeOfXAlt}.
Suppose $u$ is a pointwise minimiser of $\F(\cdot)$ in the class $W^{1,p}_g(\Omega)$. Then $u\in W^{1,q}(\Omega)$. Moreover for $0\leq \beta< \alpha$ there is $\gamma>0$ such that
\begin{align*}
\|u\|_{W^{1,\frac{np}{n-\beta}}(\Omega)}\lesssim \left(1+\F(u)+\|g\|_{W^{1+\max\left(\alpha,\frac 1 q\right),q}(\Omega)}+\|f\|_{L^{q'}(\Omega)}\right)^\gamma
\end{align*}
\end{theorem}

The proof of this theorem employs regular approximations to $u\in W^{1,p}_g(\Omega)$ using a partition of unity adapted to a Whitney-Besicovitch covering of $\Omega$. 
This construction is motivated by the following considerations: For autonomous convex functionals it is easy to see that considering the mollified functions $u\star \phi_\e$ shows that $\mathscr L(\cdot,X,Y)=0$  on $X=W^{1,p}(\Omega)$ where $Y=W^{1,q}(\omega)$ for some fixed $\omega\Subset\Omega$.
Our situation, where we intend to employ a similar argument, differs in two aspects: we consider non-autonomous functionals and require results that are valid up to the boundary. We use assumption \eqref{def:changeOfXAlt} to deal with the $x$-dependence. In order to obtain results valid up to the boundary our main idea is to use a two-parameter mollification $u\star \phi_{\e \delta(x)}$ where $\delta(x)\sim d(x,\p\Omega)$. We will implement this using a Whitney-Besicovitch covering.

We are not able to prove that $\mathscr L(u,W^{1,p}_g(\Omega),W^{1,q}_g(\Omega))=0$ for all $u\in W^{1,p}_g(\Omega)$ in the full range of $q$ covered by Theorem \ref{thm:nonautonReg}, but obtain the following partial result:
\begin{proposition}\label{cor:lavrentiev}
Suppose $\Omega$ is a Lipschitz domain and $g\in W^{1+\frac 1 q,q}(\Omega)$.
Let $X=W^{1,p}_g(\Omega)$ endowed with the weak topology.
Suppose that ${1<p\leq q<\frac{(n+\alpha)p}{n}}$ and that $F(x,\cdot)$ satisfies \eqref{def:bounds1}-\eqref{def:bounds3} and \eqref{def:changeOfXAlt}.  Then with the choice $Y=W^{1,q}_{\tp{loc}}(\Omega)$,  
\[\mathscr L(\cdot,X,Y)=0\text{ on } X.\]

If in fact $1<p\leq q <min\big (p+1,\big(1+\frac \alpha {(\alpha+1)n}\big) p\big)$, then $\mathscr L(\cdot,X,Y)=0$ on $X$ with the choice $Y=W^{1,q}_g(\Omega)$.
\end{proposition}

We mention that as a by-product of our work we also obtain a higher differentiability result:
\begin{corollary}\label{cor:higherDiff}
Suppose the assumptions of Theorem \ref{thm:regularityRelaxed} (or Theorem \ref{thm:nonautonReg}) hold. Let $u$ be a relaxed (pointwise) minimiser of $\F(\cdot)$ in the class $W^{1,p}_g(\Omega)$. Then $u\in W^{1+\beta,p}(\Omega)$ for $\beta < \frac \alpha {\max(2,p)}$. Moreover there is $\gamma>0$ such that
\begin{align*}
\|u\|_{W^{1+\beta,p}(\Omega)}\lesssim \left(1+\F(u)+\|g\|_{W^{1+\max\left(\alpha,\frac 1 q\right),q}(\Omega)}+\|f\|_{L^{q'}(\Omega)}\right)^\gamma
\end{align*}
\end{corollary}

The structure of the paper is as follows: In Section \ref{sec:prelim} we collect some background results that we use to prove global higher integrability for minimisers of the relaxed functional in Section \ref{sec:relaxed} and for minimisers of the pointwise functional in Section \ref{sec:WB}. In Section \ref{sec:alternatives} we extend these results to ellipticity assumptions suited for anisotropic and $p(x)$-growth. In Section \ref{sec:examples} we give examples of functionals to which our results apply.

\section{Preliminaries}\label{sec:prelim}
\subsection{Notation}
In this section we introduce our notation. 
The set $\Omega$ will always denote a open, bounded domain in $\R^n$. Given a set $\omega\subset\R^n$, $\overline \omega$ will denote its closure. We write $B_r(x)$ for the usual open Euclidean balls in $\R^n$ and $S^{n-1}$ for the unit sphere in $\R^n$.
We denote the cone of height $\rho$, aperture $\theta$ and axis in direction $\pmb n$ by $C_\rho(\theta,\pmb n)$. That is
\begin{align*}
C_\rho(\theta,\pmb n)=\{\,h\in \R^n: |h|\leq \rho, h\cdot \pmb n\geq |h|\cos(\theta)\,\}.
\end{align*}
Here $|\cdot |$ denotes the Euclidean norm of a vector in $\R^n$ and likewise the Euclidean norm of a matrix $A\in \R^{n\times n}$. $\tp{Id}$ denotes the identity matrix in $\R^{n\times n}$.
Given an open set $\Omega$ we denote $\Omega_\lambda=\{\,x\in \Omega: d(x,\partial\Omega)>\lambda\,\}$ and $\lambda \Omega = \{\,\lambda x: x\in\Omega\,\}$. Here $d(x,\partial\Omega)=\inf_{y\in \p\Omega} |x-y|$ denotes the distance of $x$ from the boundary of $\Omega$.

If $p\in[1,\infty]$ denote by $p'=\frac{p}{p-1}$ its H\"older conjugate.
The symbols $a \approx b$ and $a\lesssim b$ mean that there is some constant $C>0$, depending only on $n,m,p,\Omega,\mu,\nu$ and $\Lambda$, and independent of $a$ and $b$ such that $C^{-1} a \leq b \leq C a$ and $a\leq C b$, respectively. 

Write $V_{\mu,t}(z)=(\mu^2+|z|^2)^\frac{t-2}{4}z$. We recall the useful well-known inequality:
\begin{lemma}\label{lem:Vfunctional}
For $s>-1, \mu\in[0,1]$, $z_1,z_2\in\R^N$, with $\mu+|z_1|+|z_2|>0$,
\begin{align*}
\int_0^1 (\mu^2+|z_1+\lambda(z_2-z_1)^2)^\frac s 2 \lambda\d \lambda\sim (\mu^2+|z_1|^2+|z_2|^2)^\frac s 2
\end{align*}
with the implicit constants only depending on $s$.
Further,
\begin{align*}
|V_{\mu,t}(z_1)-V_{\mu,t}(z_2)|\sim (\mu^2+|z_1|^2+|z_2|^2)^\frac{t-2}{2}|z_1-z_2|^2.
\end{align*}
\end{lemma}
We will often find it useful to write for a function $v$ defined on $\R^n$ and a vector $h\in \R^n$, $v_h(x)=v(x+h)$.

We pick a family $\{\,\phi_\e\,\}$ of radially symmetric, non-negative mollifiers of unitary mass. We denote convolution with $\phi_\e$ as
\begin{align*}
u\star\phi_\e(x)=\int_{\R^n} u(y)\phi_\e(x-y)\d y.
\end{align*}

\subsection{Function spaces}
\label{sec:besov}
We recall some basic properties of Sobolev and Besov spaces following the exposition in \cite{Savare1998}, alternatively the theory can be found in \cite{Triebel1978}.

For $0\leq \alpha\leq 1$ and $k\in \N$, $C^k(\Omega)$ and $C^{k,\alpha}(\Omega)$ denote the spaces of functions $k$-times continuously differentiable in $\Omega$ and $k$-times $\alpha$-H\"older differentiable in $\Omega$, respectively.

For $1\leq p\leq\infty, k\in \N$, $L^p(\Omega)=L^p(\Omega,\R^m)$ and $W^{k,p}(\Omega)=W^{k,p}(\Omega,\R^m)$ denote the usual Lebesgue and Sobolev spaces respectively. We write $W^{k,p}_0(\Omega)$ for the closure of $C_0^\infty(\Omega)$-functions with respect to the $W^{k,p}$-norm. For $g\in W^{1,1}(\Omega)$, write $W^{k,p}_g(\Omega)= g + W^{k,p}_0(\Omega)$. We freely identify $W^{k,p}$-functions with their precise representatives. 

 Denote by $[,]_{s,q}$ the real interpolation functor. Let $s\in(0,1)$ and $p,q\in [1,\infty]$. We define
\begin{gather*}
B^{s,p}_q(\Omega)=B^{s,p}_q(\Omega,\R^m)=[W^{1,p}(\Omega,\R^m),L^p(\Omega,\R^m)]_{s,q}\\
B^{1+s,p}_q(\Omega)=[W^{2,p}(\Omega),W^{1,p}(\Omega)]_{s,q}=\{\,v\in W^{1,p}(\Omega): Dv\in B^{s,p}_q(\Omega)\,\}
\end{gather*}
Further we recall that $W^{1+s,p}(\Omega)=B^{1+s,p}_p(\Omega)$ and that for $1\leq q<\infty$, $B^{s,p}_q(\Omega)$ embeds continuously in $B^{s,p}_\infty(\Omega)$.
We will use a characterisation of these spaces in terms of difference quotients.  Let $D$ be a set generating $\R^n$, star-shaped with respect to $0$. For $s\in(0,1)$, $p\in[1,\infty]$, consider
\begin{align*}
[v]_{s,p,\Omega}^p = \sup_{h\in D\setminus\{\,0\,\}}\int_{\Omega_h}\left|\frac{v_h(x)-v(x)}{h}\right|^p\d x.
\end{align*}
This characterises $B^{s,p}_\infty(\Omega)$ in the sense that
\begin{align*}
v\in B^{s,p}_\infty(\Omega)\Leftrightarrow v\in L^p(\Omega) \text{ and } [v]_{s,p,\Omega}^p<\infty.
\end{align*}
Moreover there are positive constants $C_1,C_2>0$ depending only on $s,p,D,\Omega$ such that
\begin{align}\label{eq:besovcharacterisation}
C_1 \|v\|_{B^{s,p}_\infty(\Omega)}\leq \|v\|_{L^p(\Omega)}+[v]_{s,p,\Omega}\leq C_2\|v\|_{B^{s,p}_\infty(\Omega)}.
\end{align}
If $\Omega=B_r(x_0)$, then $C_1, C_2$ are unchanged by replacing $D$ with $QD$, where $Q$ is an orthonormal matrix. In particular, when $D=C_\rho(\theta,\pmb n)$ is a cone, they are independent of the choice of $\pmb n$. 

Finally, recall that $B^{s,p}_q(\Omega)$ may be localised: If $\{\,U_i\,\}_{i\leq M}$ is a finite collection of balls covering $\Omega$, then $v\in B^{s,p}_q(\Omega)$ if and only if $v_{|\Omega \cap U_i}\in B^{s,p}_q(\Omega \cap U_i)$ for $i=1,...,M$. Moreover, there are constants $C_3,C_4$ depending only on $M$ such that
\begin{align}\label{eq:besovlocalisation}
C_3 \|v\|_{B^{s,p}_q(\Omega)}\leq \sum_{i=1}^M \|v\|_{B^{s,p}_q(\Omega \cap U_i)}\leq C_4 \|v\|_{B^{s,p}_q(\Omega)}.
\end{align}

We recall a well-known embedding theorem, see e.g. \cite{Triebel2002}:
\begin{theorem}\label{thm:embedding}
Let $0< s\leq 1$ and $p,p_1\in [1,\infty]$.
Assume that $s-\frac n p =-\frac n {p_1}$ and suppose $v\in B^{s,p}_\infty(\Omega)$. Then for any $\e\in(0,1-p_1]$,
\begin{align*}
\|v\|_{L^{p_1-\e}(\Omega)}\lesssim\|v\|_{B^{s,p}_\infty(\Omega)}.
\end{align*}
\end{theorem}

We next recall some basic properties of mollifications. Proofs can be found in \cite{evans}.
\begin{lemma}\label{lem:mollifier}
Let $1<p<\infty$. Let $\Omega$ be an open domain in $\R^n$ and $\omega\Subset\Omega$. Suppose $u\in W^{1,p}(\Omega)$ and consider $u_\e = u\star \phi_\e$ for $\e<d(\omega,\p\Omega)$. Then
\begin{enumerate}
\item $u_\e\in C^\infty(\omega)$
\item $\|u_\e\|_{W^{1,p}(\omega)} \leq \|u\|_{W^{1,p}(\Omega)}$
\item For $p\leq q\leq \frac{np}{n-p}$, $\|u_\e-u\|_{L^p(\omega)}\lesssim \e^\Theta \|u\|_{W^{1,p}(\Omega)}$ where $\Theta = 1+n(1/q-1/p)$ if $p<n$, $\Theta = n/q$ if $p\geq n$.
\item $\|Du_\e\|_{L^\infty(\omega)}\lesssim \e^{-\frac n p}\|Du\|_{W^{1,p}(\Omega)}.$
\end{enumerate}
\end{lemma}

We also recall the trace theorem in the following form, see e.g. \cite{Evans1992}.
\begin{lemma}\label{lem:traceTheorem}
Let $\Omega$ be a Lipschitz domain. Let $1<p<\infty$. Then there is a bounded linear operator $\tp{Tr}\colon W^{1+\frac 1 p,p}\Omega\to W^{1,p}(\p\Omega)$. In fact, it is possible to define $\tp{Tr} u$ to be the values of the precise representative of $u$ on $\p\Omega$.
\end{lemma}

Finally we recall the following well-known result (see \cite{Evans1992} for the ingredients of the proof) which will justify extending $u$ by extensions of $g$.
\begin{lemma}\label{lem:extension}
Let $p\in[1,\infty]$. Let $V\Supset \Omega$ be an open, bounded set.
Suppose $u\in W^{1,p}(\Omega)$ and $v\in (u+W^{1,p}_0(\Omega))\cap W^{1,p}(V)$. Then the map
\begin{align*}
w = \begin{cases}
	u \text{ in } \Omega \\
	v \text{ in } V\setminus\Omega
	\end{cases}
\end{align*}
	belongs to $W^{1,p}(V)$.
\end{lemma}
\subsection{Some properties of Lipschitz and \texorpdfstring{$C^{1,\alpha}$}-domains}
\label{sec:lipschitzDomains}
In this section we recall some properties of Lipschitz and $C^{1,\alpha}$-domains. For further details we refer to \cite{Grisvard1992}.
We say $\Omega\subset\R^n$ is a Lipschitz ($C^{1,\alpha}$) domain if $\Omega$ is an open subset of $\R^n$ and for every $x\in \p\Omega$, there exists a neighbourhood $V$ of $x$ in $\R^n$ and orthogonal coordinates $\{\,y_i\,\}_{1\leq i\leq n}$ such that
\begin{enumerate}
\item $V$ is a hypercube in the new coordinates:
\begin{equation*}
V = \{\,(y_1,\ldots, y_n\,\}: -a_i<y_i<a_i, 1\leq i\leq n-1\,\}.
\end{equation*}
\item there exists a Lipschitz ($C^{1,\alpha}$) function $\phi$ defined in 
\begin{align*}
V'=\{\,(y_1,\ldots y_{n-1}): -a_i<y_i<a_i,1\leq i\leq n-1\,\}
\end{align*}
and such that
\begin{gather*}
|\phi(y')|\leq a_n/2 \text{ for every } y'=(y_1,\ldots, y_{n-1})\in V',\\[3pt]
\Omega\cap V = \{\,y = (y',y_n)\in V: y_n<\phi(y')\,\},\\[3pt]
\p\Omega \cap V = \{\,y=(y',y_n)\in V: y_n = \phi(y')\,\}.
\end{gather*}
\end{enumerate}

Let $\Omega$ be a Lipschitz domain. Then $\Omega$ satisfies a uniform exterior cone condition \cite[Section 1.2.2]{Grisvard1992}: there are $\rho_0,\theta_0>0$ and a map $\pmb n\colon \R^n\to S^{n-1}$ such that for every $x\in \R^n$
\begin{align}\label{eq:uniformCone}
C_{\rho_0}(\theta_0,\pmb n(x))\subset O_{\rho_0}(x)=\left\{\,h\in\R^n: |h|\leq \rho_0, \left(\Omega\setminus B_{3\rho_0}(x)\right)+h\subset\R^n\setminus\Omega\,\right\}.
\end{align}
Moreover there is a smooth vector field transversal to $\p\Omega$, i.e. there exists $\kappa>0$ and ${X\in C^\infty(\R^n,\R^n)}$ such that
\begin{align*}
X\cdot \nu\geq \kappa
\end{align*}
a.e. on $\p\Omega$ where $\nu$ is the exterior unit normal to $\p\Omega$ \cite[Lemma 1.5.1.9.]{Grisvard1992}.

It is well-known that every Lipschitz domain can be written as a finite union of strongly star-shaped Lipschitz domains:
\begin{lemma}\label{lem:starshapedLipschitz}
Suppose $\Omega$ is a Lipschitz domain. Then there is $N>0$ and strongly starshaped Lipschitz domains $\Omega_i,\omega_i$, $i=1,\ldots N$ 
such that $\omega_i\Subset\Omega_i$ relative to $\Omega$, ${\p\omega_i\cap\p\Omega \Subset \p\Omega_i\cap\p\Omega}$ if $\p\Omega_i\cap\Omega\neq\emptyset$ and
\[\Omega = \bigcup_{i=1}^N \Omega_i=\bigcup_{i=1}^N \omega_i.\]
Moreover given $g\in W^{1+\frac 1 q,q}(\Omega)$ and $u\in W^{1,p}_g(\Omega)$, we can ensure that $u\in W^{1,p}(\p\Omega_i)$ for $i=1,\ldots,N$ with $\|u\|_{W^{1,p}(\p\Omega_i)}\lesssim \|u\|_{W^{1,p}(\Omega_i)}+\|g\|_{W^{1+1/q,q}(\Omega)}$.
\end{lemma}
\begin{proof}
We recall the main points of the proof and refer to \cite[Proposition 2.5.3]{Carbone2019} for details.
Let $x\in\p\Omega$. Then for $\e$ and $\rho$ sufficiently small we may assume that $I_x$ is a neighbourhood of $x$ such that, writing $\hat y = (y_1,\ldots,y_{n-1})$,
\begin{align*}
I_x\cap \overline\Omega = \{\,y\in \R^n: -\e<y_n\leq \tau(\hat y),\qquad \hat y \in B_\rho(0)\,\}
\end{align*}
where $\tau$ is Lipschitz function. We may further assume that $\rho\leq \frac{\tau(0)}{2c}$ where $c$ denotes the Lipschitz constant of $\tau$. We now consider 
\[
\Omega_x^{\e',\rho'} = \{\,y\in \R^n : -\e'<y_n<\tau(\hat y), \qquad \hat y \in B_{\rho'}(0)\,\}
\]
 with $\e'\in(0,\e)$, $\rho'\in (0,\rho)$.
 A straightforward argument shows that $\Omega^{\e',\rho'}$ is strongly star-shaped with respect to $0$ as well as a Lipschitz domain.
Since
\begin{align*}
\int_0^\e\int_0^\rho \int_{\p\Omega_x^{\e',\rho'}} (|u|^p+|Du|^p)\lesssim \int_{I_x\cap\Omega} |Du|^p\d x
\end{align*}
we can choose $\tilde\e\in(0,\e)$ and $\tilde\rho\in(0,\rho)$ such that 
\begin{align*}
\int_{\p\Omega_x^{\e',\rho'}\setminus\p\Omega}|u|^p+|Du|^p\lesssim c(\e,\rho)\|u\|_{W^{1,p}(\Omega)}.
\end{align*}
We now set $\Omega_x = \Omega_x^{\tilde\e,\tilde\rho}$ and $\omega_x = \Omega_x^{\tilde \e/2,\tilde\rho/2}$.
If $x\in\Omega$, by a similar argument we find a ball $B_{\tilde \e}(x)$ with $\tilde\e\in(0,\p\Omega/2)$ such that $\|u\|_{W^{1,p}(\p B_{\tilde \e}(x))}\lesssim c(d(x,\p\Omega))\|u\|_{W^{1,p}(B_{\tilde \e}(x))}$ and set $\Omega_x = B_{\tilde \e}(x)$, $\omega_x = B_{\tilde \e/2}(x)$. A standard compactness argument shows that the cover $\{\,\omega_x\,\}_{x\in\Omega}$ of $\Omega$ contains a finite subcover with the desired properties.
\end{proof}

The next lemma enables us to to stretch a small neighbourhood of the boundary in a controlled manner. This will be crucial in constructing sequences with improved integrability but unchanged boundary behaviour.
\begin{lemma}\label{lem:diffeos}
Suppose $\Omega$ is a $C^{1,\alpha}$-domain. Then
there is a family of domains $\Omega^s\Supset\Omega$ and a family of $C^{1,\alpha}$-diffeomorphisms $\Psi_s\colon \Omega^s\to\Omega$ such that
\begin{enumerate}
\item $J\Psi_s\to 1$ and $|D\Psi_s-\tp{Id}|\to 0$ uniformly in $\Omega^s$ as $s\nearrow 1$. Equivalently, $J\Psi_s^{-1}\to 1$ and $|D\Psi_s^{-1}-\tp{Id}|\to 0$ uniformly in $\Omega$ as $s\nearrow 1$.
\item If $g\in W^{1+\frac 1 q,q}(\Omega)$ there is an extension $\hat g$ of $g$ to $\Omega^s$ such that $\hat g\in W^{1,q}(\Omega^s)$ and $\hat g\circ \Psi_s^{-1}\in g+W^{1,q}(\Omega)$.
\end{enumerate}
\end{lemma}
\begin{proof}
Let $X\in C^\infty(\R^n,\R^n)$ be 
a smooth vector field transversal to $\p\Omega$. Fix $t_0>0$.
Given $z\in\p\Omega$ and for $|t|\leq 2t_0$ consider the flow
\begin{align*}
\frac{d h_z}{d t} =& X(h(t))\\
h_z(s)=& z
\end{align*}
and set $\Psi(z,t)=h_z(t)$. After possibly reducing the value of $t_0$, the maps $\Psi,\Psi^{-1}$ are $C^{1,\alpha}$-regular diffeomorphisms on a neighbourhood of $\p\Omega$ which we denote $V$. Moreover the Jacobians of $\Psi$, $\Psi^{-1}$ are bounded. 

Let $\frac 1 2\leq s<1$. Consider ${\tau_s\colon [-t_0,t_0]\to [-t_0,t_0]}$, a sequence of strictly monotonically increasing smooth maps with
\begin{align*}
\tau_s(-t_0)=-t & & \tau_s(t_0)=s^{-1} t_0 & &\tau_s'(-t_0)=-1
\end{align*}
and such that $\tau_s\to \tp{Id}$ in $C^2([-t_0,t_0])$ as $s\to 1$. Define
\begin{align*}
\Psi_s^{-1}(x)=\begin{cases}
			\Psi\left(x_0,\tau_s(t)\right) &\text{ for } x=\Psi(x_0,t)\in V \\
			x &\text{ else }. 
		  \end{cases}
\end{align*}
Set $\Omega^s = \Psi_s^{-1}(\Omega)\subset V\cup \Omega$.
Using the chain rule we note that $\Psi_s^{-1}\colon \Omega\to \Omega^s$ is a $C^{1,\alpha}$-regular diffeomorphism. Denote its inverse by $\Psi_s\colon \Omega^s\to\Omega$ and note that using the Inverse Function Theorem and the chain rule $\Psi_s\to \tp{Id}$ in $C^{1}(\Omega^s)$ as $s\nearrow 1$. In particular, also $J\Psi_s\to 1$ uniformly in $\Omega^s$ as $s\nearrow 1$. To obtain $\hat g$, simply set 
\begin{align*}
\hat g(x)=\begin{cases}
			g(x_0,0) &\text{ for } x=\Psi(x_0,t_0)\in \Omega^s\setminus\Omega \\
			g(x) &\text{ for } x\in \Omega. 
		  \end{cases}
\end{align*}
\end{proof}

We conclude this section by noting a number of extensions we may carry out if $\Omega$ is a Lipschitz domain:
Let $\Omega\Subset B(0,R)$.
Due to \cite{Rychkov1999}, if $g\in W^{s,p}(\Omega)$, then there is an extension ${\tilde g\in W^{s,p}(\R^n,\R^m)}$ with
\begin{align}\label{eq:Gextension}
\|\tilde g\|_{W^{s,p}(\R^n)}\lesssim \|g\|_{W^{s,p}(\Omega)}.
\end{align}
Further we extend $F(x,z)$ to a functional on $B(0,R) \times\R^{n\times m}$, still denoted $F(x,z)$, that satisfies
\begin{gather*}
|F(x,z)-F(y,z)|\lesssim \Lambda |x-y|^\alpha (1+|z|^2)^\frac q 2\\
|F(x,z)|\lesssim \Lambda(1+|z|^2)^\frac q 2
\end{gather*}
by setting for $x\in B(0,R)\setminus\Omega$,
\begin{align*}
F(x,z)= \inf_{y\in\Omega}\left(F(y,z)+\Lambda \left(1+|z|^2\right)^\frac q 2 |x-y|\right).
\end{align*}

\subsection{Whitney-Besicovitch coverings}
In this section we define Whitney-Besicovitch coverings which will be a key tool in our proof of regularity for pointwise minimisers. Whitney-Besicovitch coverings combine properties of Whitney and Besicovitch coverings and were introduced in \cite{Kislyakov2005}. A nice presentation of the theory is given in \cite{Kislyakov2013}. To be precise:
\begin{definition}
A family of dyadic cubes $\{\,Q_i\,\}_{i\in I}$ with mutually disjoint interiors is called a Whitney-covering of $\Omega$ if
\begin{align*}
\bigcup_{i\in I} Q_i = \bigcup_{i\in I} 2Q_i = \Omega\\
5Q_i \cap (\R^n\setminus \Omega) \neq \emptyset.
\end{align*}
 A family of cubes $\{\,K_i\,\}_{i\in I}$ is called a Whitney-Besicovitch-covering (WB-covering) of $\Omega$ if there is a triple $(\delta, M, \varepsilon)$ of positive numbers such that
 \begin{align}\label{def:WBcoverExtension}
  & \bigcup_{i\in I} \frac{1}{1+\delta}K_i = \bigcup_{i\in I} K_i = \Omega\\\label{def:WBcoverMultiplicity}
  & \sum_{i\in I} \chi_{K_i}\leq M\\
  & K_i \cap K_j \neq \emptyset \Rightarrow |K_i \cap K_j| \geq \varepsilon \max(|K_i|,|K_j|).  
 \end{align}

\end{definition}
The existence of a Whitney-covering for $\Omega$ is classical. The refinement to a WB-covering can be found in \cite{Kislyakov2013}:
\begin{theorem}[cf. Theorem 3.15, \cite{Kislyakov2013}]\label{thm:covering}
Let $\Omega$ be an open subset of $\R^n$ with non-empty complement. Let $\{\,Q_i\,\}$ be a family of cubes which are a Whitney covering of $\Omega$. Then the cubes $K_i = \left(1+\frac{1}{6}\right)Q_i$ are a WB-covering of $\Omega$ with $\delta = \frac{1}{6}$, $\e = \frac{1}{14^n}$ and $M\leq 6^n-4^n+1$. Moreover for this covering $\frac{2}{\left(1+\frac{1}{6}\right)^\frac{1}{n}}|K_i|^\frac{1}{n}\leq\dist(K_i,\partial\Omega)$.
\end{theorem}
It will be of crucial importance to us that there exists a partition of unity associated to a WB-covering.
\begin{theorem}[cf. Theorem 3.19, \cite{Kislyakov2013}]\label{thm:WBunity}
 Suppose the cubes $\{\,K_i\,\}_{i\in I}$ form a WB-covering of $\Omega$ with constants $(\delta,M,\varepsilon)$. Then there is a family $\{\,\psi_i\,\}_{i\in I}$ of infinitely differentiable functions that form a partition of unity on $\Omega$ with the following properties:
 \begin{align*}
  &\supp (\psi_i) \subset \frac{1+\frac{\delta}{2}}{1+\delta} K_i\\
  &\psi_i(x) \geq \frac{1}{M} \text{ for } x\in \frac{1}{1+\delta} K_i\\
  &|D\psi_i| \leq c \frac{1}{|K_i|^{\frac{1}{n}}}.
 \end{align*}
 for all $i\in I$.
\end{theorem}

\subsection{Properties of the relaxed and the regularised functional}
In this section, we collect some results regarding the relaxed function $\overline \F(\cdot)$ and its relation to a regularised version of $\F(\cdot)$. For local versions of $\overline \F(\cdot)$ these results are well-known, c.f. Section 6 in \cite{Marcellini1989} and \cite{Esposito2004}.
Define for $u\in W^{1,p}(\Omega)$
\begin{align*}
\F_\e(u)= \begin{cases}
			\F(u)+\e \int_\Omega |Du|^q\d x &\text{ if } u\in W^{1,q}(\Omega)\\
			\infty &\text{ else}.
			\end{cases}
\end{align*}
Write $F_\e(x,z)=F(x,z)+\e|z|^q$. If $F(x,z)$ is measurable in $x$ and continuously differentiable in $z$, it is easy to see that so is $F_\e(x,z)$. Moreover if $F(x,z)$ satisfies \eqref{def:bounds1}-\eqref{def:bounds3}, then $F_\e(x,z)$ satisfies the bounds
\begin{gather}\label{eq:approxFelliptic}
\nu(\mu^2+|z|^2+|w|^2)^\frac{p-2}{2}\leq\frac{F_\e(x,z)-F_\e(x,w)-\p_z F_\e(x,w)\cdot (z-w)}{|z-w|^2}\\\label{eq:approxFupper}
 |F_\e(x,z)|\leq (\Lambda+\e) (1+|z|^2)^\frac{q}{2} \\
 |\p_z F_\e(x,z)|\lesssim (\Lambda+\e) (1+|z|^2)^\frac{q-1}{2}\label{eq:partialFupper}\\
 \label{eq:approxFholder}
 |F_\e(x,z)-F_\e(y,z)|\leq \Lambda|x-y|^\alpha(1^2+|z|^2)^\frac{q}{2}.
\end{gather}
Further note that if $F(x,z)$ satisfies \eqref{def:changeOfXAlt}, then so does $F_\e(x,z)$.

Minimisers of $\F_\e(\cdot)$ and $\overline{\F}(\cdot)$ are related as follows:
\begin{lemma}[c.f. Lemma 6.4. in \cite{Marcellini1989}]\label{lem:convApproximate} Let $g\in W^{1,q}(\Omega)$. Suppose $F(x,z)$ satisfies \eqref{def:bounds1} and \eqref{def:bounds2}.
Suppose $u$ is a relaxed minimiser of $\F(\cdot)$ in the class $W^{1,p}_g(\Omega)$ and $u_\e$ is the pointwise minimiser of $\F_\e(\cdot)$ in the class $W^{1,q}_g(\Omega)$. Then $\F_\e(u_\e)\to \overline\F(u)$ as $\e\to 0$. Moreover up to passing to a subsequence $u_\e\to u$ in $W^{1,p}(\Omega)$.
\end{lemma}
\begin{proof}
Note that by our assumptions $f\cdot u\in L^1(\Omega)$.
Hence existence and uniqueness of $u_\e$ follows from the direct method and strict convexity, respectively.
We further note that 
\begin{align*}
\overline\F(u)\leq \liminf_{\e\to 0}\F(u_\e) \leq \liminf_{\e\to 0}\F_\e(u_\e).
\end{align*}
 To prove the reverse implication note that for any ${v\in W^{1,q}_g(\Omega)}$,
\begin{align*}
\limsup_{\e \to 0} \F_\e(u_\e) \leq \lim_{\e \to 0} \F_\e(v)= \F(v)= \overline \F(v).
\end{align*}
By definition of $\overline\F(\cdot)$ the inequality above extends to all $v\in W^{1,p}_g(\Omega)$. In particular, it holds with the choice $v=u$. Thus $\F_\e(u_\e)\to \overline \F(u)$.

Using \eqref{def:bounds1} we may extract a (non-relabelled) subsequence of $u_\e$ so that $u_\e\rightharpoonup v$ weakly in $W^{1,p}(\Omega)$ for some $v\in W^{1,p}(\Omega)$. Note using our calculations above that $v$ is a relaxed minimiser of $\F(\cdot)$ in the class $W^{1,p}_g(\Omega)$. Using $\eqref{def:bounds1}$ it is easy to see that for $w_1,w_2\in W^{1,q}_g(\Omega)$,
\begin{align}\label{eq:convexArg}
\overline \F\left(\frac{w_1+w_2}{2}\right)+\frac \nu p\|Dw_1-Dw_2\|_{L^p(\Omega)}^p\leq \frac 1 2\left(\overline \F(w_1)+\overline \F(w_2)\right).
\end{align}
Using the definition of $\overline \F(\cdot)$ and weak lower semicontinuity of norms, we see that this estimate extends to ${w_1,w_2\in W^{1,p}_g(\Omega)}$. In particular, $\overline \F(\cdot)$ is convex and so $u=v$. Moreover the choice $w_1=u,w_2=u_\e$ in the estimate shows that $u_\e\to u$ in $W^{1,p}(\Omega)$.
\end{proof}

A useful criterion for establishing the equality $\overline \F(u)=\F(u)$ is the following:
\begin{lemma}[cf. \cite{Buttazo1995}]\label{lem:lavrientiev}Let $1< p\leq q<\infty$.
For $u\in W^{1,p}_g(\Omega)$ with $\F(u)<\infty$, we have $\overline \F(u)=\F(u)$ if and only if there is a sequence $u_k\in W^{1,q}_g(\Omega)$ such that $u_k\rightharpoonup u$ weakly in $W^{1,p}(\Omega)$ and $\F(u_\e)\to \F(u)$ as $\e\to 0$. 
\end{lemma}

We close this section by showing that relaxed minimisers are very weak solutions of the Euler-Lagrange system. We first recall a Lemma from \cite{DeFilippis2020}.
\begin{lemma}\label{lem:improvedLowerBound}
Suppose $F(x,z)$ satisfies \eqref{def:bounds1} and \eqref{def:bounds2}. Then
\begin{align*}
\p_z F(z)\cdot z\gtrsim |z|^q+|\p_z F(z)|^{q'}-1.
\end{align*}
\end{lemma}
\begin{proof}
Consider the Fenchel-conjugate of the partial integrand $F(x,z)$,
$$F^*(x,\xi)=\sup_{\zeta\in \R^{n\times m}} \left(\xi\cdot\zeta-F(x,\zeta)\right)\qquad (x,\xi)\in \Omega\times \R^{n\times m}.$$
The Fenchel-Young inequality in the extremal case yields that
\begin{align*}
\p_z F(z)\cdot z = F(x,z)+F^*(z,\p_z F(x,z)).
\end{align*}
To conclude note by direct calculation that, $F^*(x,\p_z F(x,z))\gtrsim |\p_z F(x,\xi)|^{q'}-1$.
\end{proof}

We adapt an argument from \cite{DeFilippis2020}.
\begin{lemma}\label{lem:relaxedEuler}
Let $g\in W^{1,q}(\Omega)$ and suppose $F(x,z)$ satisfies \eqref{def:bounds1} and \eqref{def:bounds2}. Let $u$ be a relaxed minimiser of $\F(\cdot)$ in the class $W^{1,p}_g(\Omega)$. Then $\p_z F(x,Du)\in L^{q'}(\Omega)$ and
\begin{align*}
\int_\Omega \p_z F(x,Du)\cdot D\psi-f\psi\d x=0\qquad \forall \psi\in W^{1,q}_0(\Omega).
\end{align*}
\end{lemma}
\begin{proof}
By the direct method we obtain $u_\e \in W^{1,q}_g(\Omega)$ minimising $\F_\e(\cdot)$ in the class $W^{1,q}_g(\Omega)$. Denote $\sigma_\e = \p_z F_\e(\cdot,Du_\e)$ and $\mu_\e = |Du_\e|^{q-2}Du_\e$. Then $u_\e$ satisfies the Euler-Lagrange system
\begin{align*}
\int_\Omega (\sigma_\e+\e\mu_\e)\cdot D\psi-f\psi\d x=0 \qquad \forall \psi\in W^{1,q}_0(\Omega).
\end{align*}
Choose $\psi = u_\e-g$ and use Lemma \ref{lem:improvedLowerBound}, H\"older and Young's inequality to find that for any $\delta>0$,
\begin{align*}
\int_\Omega |Du_\e|^p+|\sigma_\e|^{q'}+\e |Du_\e|^q\d x\lesssim & C\int_\Omega (\sigma_\e+\e \mu_\e)\cdot Du_\e\d x\\
=& C\int_\Omega (\sigma_\e + \e \mu_\e)\cdot Dg+f\cdot (u_\e-g)\d x\\
\leq& \delta \left(\|\sigma_\e\|_{L^{q'}(\Omega)}^{q'}+\|u_\e\|_{L^{q}(\Omega)}^q+\e \|Du_\e\|_{L^q(\Omega)}^q\right)\\
&\qquad+C(\delta)\left(1+\|g\|_{L^q(\Omega)}^q+\|Dg\|_{L^q(\Omega)}^q+\|f\|_{L^{q'}(\Omega)}^{q'}\right)
\end{align*}
Noting that $\|u_\e\|_{L^q(\Omega)}\lesssim \|Du_\e\|_{L^p(\Omega)}$ by our choice of $q$, after choosing $\delta$ sufficiently small we may rearrange to conclude
\begin{align*}
\limsup_{\e\searrow 0}\int_\Omega |\sigma_\e|^{q'}\d x\lesssim 1+\int_\Omega |f|^{q'}+|g|^q+|Dg|^q\d x.
\end{align*}
Since $\p_z F(x,z)$ is Carath\'{e}odory and $Du_\e\to Du$ in measure by Lemma \ref{lem:convApproximate}, we also have $\sigma_\e\to F_z(\cdot,Du)$ in measure. Finally note that $\sigma_\e + \e \mu_\e \rightharpoonup F_z(\cdot,Du)$ in $L^1(\Omega)$, so the latter is row-wise divergence free and since it is also in $L^{q'}(\Omega)$, the result is established.
\end{proof}

\section{Regularity of minimisers of the relaxed functional}\label{sec:relaxed}
The aim of this section is to prove Theorem \ref{thm:regularityRelaxed}. For the convenience of the reader we recall the statement:
\begin{theorem}\label{thm:relaxed}
Suppose $F(x,\cdot)$ satisfies \eqref{def:bounds1}-\eqref{def:bounds3} with $1< p\leq q< \frac{(n+\alpha)p}{n}$. Then if $u$ is a relaxed minimiser of $\F(\cdot)$ in the class $W^{1,p}_g(\Omega)$, $u\in W^{1,q}(\Omega)$. Moreover for any $0\leq \beta<\alpha$ there is $\gamma>0$ such that, we have the estimate
\begin{align*}
\|u\|_{W^{1,\frac{np}{n-\beta}}(\Omega)} \lesssim \left(1+\overline\F(u)+\|g\|_{W^{1+\alpha,q}(\Omega)}+\|f\|_{L^{q'}(\Omega)}\right)^\gamma.
\end{align*}
\end{theorem}

The key to proving the theorem is an a-priori estimate for minimisers of the regularised functional:
\begin{lemma}\label{lem:apriori} Suppose the assumptions of Theorem \ref{thm:regularityRelaxed} hold.
Let $v_\e$ be the minimiser of $\F_\e(\cdot)$ in the class $W^{1,q}_g(\Omega)$. Then for any $0\leq \beta<\alpha$ there is $\gamma>0$ such that the estimate
\begin{align*}
\|v_\e\|_{W^{1,\frac{np}{n-\beta}}(\Omega)}\lesssim \left(1+\F(v_\e)+\|g\|_{W^{1+\alpha,q}(\Omega)}+\|f\|_{L^{q'}(\Omega)}\right)^\gamma 
\end{align*} 
holds, with the implicit constant independent of $\e$ and $\gamma$.
\end{lemma}

Theorem \ref{thm:regularityRelaxed} is now a consequence of combining Lemma \ref{lem:apriori} and Lemma \ref{lem:convApproximate}.
\begin{proof}[Proof of Theorem \ref{thm:regularityRelaxed}]
Let $v_\e\in W^{1,q}_g(\Omega)$ be the minimiser of $\F_\e(\cdot)$ in the class $W^{1,q}_g(\Omega)$. For $0\leq \beta<\alpha$ by Lemma \ref{lem:apriori}, $\{\,v_\e\,\}$ is uniformly bounded in $W^{1,\frac{np}{n-\beta}}(\Omega)$. From Lemma \ref{lem:convApproximate} we know that $\F_\e(v_\e)\to \overline \F(u)$ and up to passing to a subsequence $v_\e\to u$ in $W^{1,p}(\Omega)$. Thus we may extract a non-relabelled subsequence such that $v_\e\rightharpoonup u$ weakly in $W^{1,\frac{np}{n-\beta}}(\Omega)$. This allows to pass to the limit in Lemma \ref{lem:apriori} to conclude.
\end{proof}
We proceed to prove Lemma \ref{lem:apriori}:
\begin{proof}[Proof of Lemma \ref{lem:apriori}]
By direct methods and strict convexity deriving from (\ref{eq:approxFelliptic}) we obtain a sequence of minimisers $v_\e\in W^{1,q}_0(\Omega)$ that solve
\begin{align*}
\min_{v\in W^{1,q}_0(\Omega)} \F_\e(v+g)
\end{align*}
Moreover $v_\e$ satisfies the Euler-Lagrange equation
\begin{align*}
\int_\Omega \left(\partial_z F_\e(x,Dv_\e+Dg)\right)\cdot D\phi-f\phi\d x=0 \qquad\forall\phi\in W^{1,q}_0(\Omega).
\end{align*} In particular, for any $v\in W^{1,q}_0(\Omega)$,
\begin{align*}
\int_\Omega \left(\partial_z F_\e(x,Dv_\e+Dg)\right)\cdot D(v-v_\e)-f\cdot(v-v_\e)=0.
\end{align*}
Using this identity, we see for $v\in W^{1,q}_0(\Omega)$,
\begin{align}\label{eq:lowerEstimate}
\F_\e(v+g)-\F_\e(v_\e+g)=&\int_\Omega F_\e(x,Dv+Dg)-F_\e(x,Dv_\e+Dg)-f\cdot(v-v_\e)\d x\nonumber\\
\qquad& -\int_\Omega \left(\partial_z F_\e(x,Dv_\e+Dg)\right)\cdot D(v-v_\e)-f\cdot(v-v_\e) \d x\nonumber\\
\gtrsim&\|V_{\mu,p}(Dv+Dg)-V_{\mu,p}(Dv_\e+Dg)\|_{L^2(\Omega)}^2,
\end{align}
where to obtain the last line we have used \eqref{eq:approxFelliptic} and   Lemma \ref{lem:Vfunctional}.

Let $\rho_0>0$ and $\pmb n\colon \R^n\to S^{n-1}$ be so that the uniform cone property \eqref{eq:uniformCone} holds. Possibly reducing $\rho_0$ assume without loss of generality that $\Omega+B_{3\rho_0}(x)\subset B(0,R)$ for all $x\in \Omega$. Here $B(0,R)\Supset\Omega$ is the ball defined in Section \ref{sec:lipschitzDomains}.
Given ${x_0\in\Omega}$, let $0\leq {\phi=\phi_{x_0,\rho_0}}\leq 1$ be a smooth cut-off supported in $B_{2\rho_0}(x_0)$ with $\phi(x)=1$ in $B_{\rho_0}(x_0)$ and ${|D^k\phi(x)|\leq C_k {\rho_0}^{-k}}$  for some $C_k>0$ and $k\in N$.
Given functions $v_1,v_2$ defined on $\R^n$ and $h\in\R^n$ introduce
\begin{align*}
T_h v_1 = \phi v_{1,h}+(1-\phi)v_1
\end{align*}
and denote $\Delta_h v_1= v_{1,h}-v_1$. If $v_1,v_2\colon \R^n\to \R^{n\times m}$ then write 
$$\tilde\Delta_h F_\e(x,v_1,v_2)=F_\e(x, v_{1,h}+v_2)-F_\e(x,v_1+v_2).$$

We claim that for every $x_0\in \R^n$ there is a constant $C=C(n,\rho_0,\Lambda,\Omega)$ such that for all $v\in W^{1,q}_0(\Omega)$
\begin{align}\label{clm:diffEstimate}
\sup_{h\in C_{\rho_0}(\theta_0,\pmb n(x_0))}\frac{\F_\e(T_h \tilde v +g)-\F_\e(v+g)}{|h|^\alpha}\leq C\left(1+\|Dv\|_{L^q(\Omega)}^q + \|g\|_{W^{1+\alpha,q}(\Omega)}^q+\|f\|_{L^{q'}(\Omega)}^{q'}\right)
\end{align}
Here $\tilde v$ is the extension by $0$ to $\R^n$ of $v$.

Let $v\in W^{1,q}_0(\Omega)$ and $h\in C_{\rho_0}(\theta_0,\pmb n(x_0))$.
Let $\tilde v$ be the extension of $v$ by $0$ to a function on $\R^n$.
We estimate for $h\in C_{\rho_0}(\theta_0,\pmb n(x_0))$,
\begin{align*}
&\F_\e(T_h \tilde v+g)-\F_\e(v+g)\\
=&\int_\Omega F_\e(x,T_h D\tilde v +D\phi(\tilde v_h-v)+Dg)-F_\e(x,T_h D\tilde v+Dg)\d x\\
&\qquad+\int_\Omega F_\e(x,T_h D\tilde v+Dg)-F_\e(x,D\tilde v+Dg)\d x-\int_\Omega f\cdot(T_h\tilde v-v)\d x\\
=& A_1 + A_2+ A_3.
\end{align*}
Using (\ref{eq:approxFupper}), H\"older's inequality and recalling that $\tilde g$ denotes a $W^{1+\alpha,q}$-extension of $g$ to $\R^n$, we find
\begin{align*}
|A_1|\leq& (\Lambda +\e)\int_\Omega |D\phi\Delta_h \tilde v|\big(1+|T_h D\tilde v+Dg|^2+|D\phi\Delta_h\tilde v|^2\big)^\frac{q-1}{2}\d x\\
\leq& \|D\phi\Delta_h\tilde v\|_{L^q(B_{2\rho_0}(x_0))}\big(1+\|T_hD\tilde v\|_{L^q(B_{2\rho_0}(x_0))}\\
&\qquad+\|D\phi\Delta_h\tilde v\|_{L^q(B_{2\rho_0}(x_0))}+\|D\tilde g\|_{L^q(B_{2\rho_0}(x_0))}^q\big)^{q-1}.
\end{align*}
Recalling that as $\tilde v\in W^{1,q}_0(\Omega,\R^m)$,
\begin{align*}
\|D\phi(\tilde v_h-\tilde v)\|_{L^q(B_{2\rho_0}(x_0))}\lesssim |h| \|D\tilde v\|_{L^q(B_{2\rho_0}(x_0))},
\end{align*}
we conclude using (\ref{eq:Gextension}) and the definition of $T_h$,
\begin{align*}
|A_1|\lesssim |h|\big(1+\|Dv\|_{L^q(\Omega)}^q+\|g\|_{W^{1+\alpha,q}(\Omega)}^q\big).
\end{align*}
We now turn to $A_2$. By convexity of $F_\e(\cdot)$,
\begin{align*}
\tilde\Delta_h F_\e(x,D\tilde v)
\leq& (1-\phi)F_\e(x,Dv+Dg)+\phi F_\e(x,D\tilde v_h+Dg)-F_\e(Dv+Dg)\\[4pt]
 =& \phi\tilde\Delta_h F_\e(x,D\tilde v,D\tilde g).
\end{align*}
In particular, using a change of coordinates and since in $B_{3\rho}\setminus \Omega$, $D\tilde v_h=D\tilde v=0$,
\begin{align*}
|A_2| \leq& \int_{B_{2\rho_0}(x_0)} \phi \tilde\Delta_h F_\e(x,D\tilde v+Dg)\d x\\
=& \int_{B_{2\rho_0}(x_0)+h}\phi(x-h)F_\e(x-h,D\tilde v+D\tilde g_{-h})\d x-\int_{B_{2\rho_0}(x_0)}\phi(x)F_\e(x,D\tilde v+D\tilde g)\d x\\
=& \int_{B_{3\rho_0}(x_0)}\Delta_{-h}\phi F_\e(x-h,D\tilde v+D\tilde g_{-h})+\phi(x)\tilde\Delta_{-h}F_\e(x-h,D\tilde g+ D\tilde v)\d x\\
&\qquad\qquad+\int_{B_{2\rho_0}(x_0)}\phi(x)\big(F_\e(x-h,D\tilde v+D\tilde g)-F_\e(x,D\tilde v+D\tilde g)\big)\d x.
\end{align*}

Using (\ref{eq:approxFupper}), \eqref{eq:approxFholder}, the regularity of $\phi$ and $g$ as well as (\ref{eq:Gextension}) to estimate each term in turn, we conclude
\begin{align*}
A_2\lesssim& |h|^\alpha\left(1+\|Dv\|_{L^q(\Omega)}^q+\|g\|_{W^{1+\alpha,q}(\Omega)}^q\right).
\end{align*}
Finally, using Young's inequality and Sobolev embedding,
\begin{align*}
|A_3|\leq& \int_\Omega |f\cdot(\phi\Delta_h v)|\lesssim \|f\|_{L^{q'}(\Omega)}\|v-v_h\|_{L^q(\Omega)}
\lesssim h\left(\|f\|_{L^{q'}(\Omega)}^{q'}+\|Dv\|_{L^q(\Omega)}^q\right)
\end{align*}
This proves the claim.

We are now ready to finish the proof. Note that $T_h\tilde v_\e\in W^{1,q}_0(\Omega)$ if $h\in C_{\rho_0}(\theta_0,\pmb n(x_0))$. Thus combining \eqref{eq:lowerEstimate} and \eqref{clm:diffEstimate} with the choice $v=T_h \tilde v_\e$ we find for $x_0\in \R^n$,
\begin{align*}
&\sup_{h\in C_{\rho_0}(\theta_0,\pmb n(x_0))}|h|^{-\alpha}\|V_{\mu,p}(Dv_\e+Dg)-V_{\mu,p}(D T_h\tilde v_{\e,h}+Dg)\|_{L^2(B_{\rho_0}(x_0))}^2\\
\lesssim& 1+\| Dw\|_{L^q(\Omega)}^q+\|g\|_{W^{1+\alpha,q}(\Omega)}^q+\|f\|_{L^{q'}(\Omega)}^{q'}\nonumber.
\end{align*}
In particular, we deduce after recalling the definition of $T_h$, using the triangle inequality and regularity of $\tilde g$,
\begin{align}\label{eq:basicBesov}
&\sup_{h\in C_{\rho_0}(\theta_0,\pmb n(x_0))}|h|^{-\alpha}\|D(\tilde v_{\e,h}+\tilde g_h)-D(v_\e+g)\|_{L^p(B_{\rho_0}(x_0)}^p\\
\lesssim& 1+\| Dw\|_{L^q(\Omega)}^q+\|g\|_{W^{1+\alpha,q}(\Omega)}^q+\|f\|_{L^{q'}(\Omega)}^{q'}\nonumber.
\end{align}
Using the characterisation of Besov spaces (\ref{eq:besovcharacterisation}), we conclude for all $x_0\in \R^n$,
\begin{align}
[Dv_\e]_{\frac \alpha p,p,B_{\rho_0}(x_0)}^p\lesssim \left(1+\|Dv_\e\|_{L^q(\Omega)}^q+\|g\|_{W^{1+\alpha,q}(\Omega)}^q+\|f\|_{L^{q'}(\Omega)}^{q'}\right)
\end{align}
Covering $\Omega$ by a finite number of balls of radius $\rho_0$ and using (\ref{eq:besovlocalisation}), we conclude
\begin{align}\label{eq:besovestimate2}
\|v_\e\|_{B^{1+\frac \alpha p,p}_\infty(\Omega)}^p\lesssim \left(1+\|v_\e\|_{W^{1,q}(\Omega)}^q+\|g\|_{W^{1+\alpha,q}(\Omega)}^q+\|f\|_{L^{q'}(\Omega)}^{q'}\right).
\end{align}
Recall now that $B^{1+\frac \alpha p,p}_\infty(\Omega)$ embeds continuously into $W^{1,\frac{np}{n-\beta}}(\Omega)$ for any $\beta <\alpha$ by Theorem \ref{thm:embedding}. Choose $\beta$ such that $q< \frac{p(\beta+n)}{n}$ and use interpolation with $\theta = \frac{np}{\beta}\left(\frac 1 p-\frac 1 q\right)$ to see
\begin{align}\label{eq:interpolation}
\|v_\e\|_{W^{1,q}(\Omega)}\leq \|v_\e\|_{W^{1,p}(\Omega)}^{1-\theta} \|v_\e\|_{W^{1,\frac{np}{n-\beta}}(\Omega)}^{\theta}
\end{align}
As, $q<\frac{(n+\beta)p}{n}$, it follows that $q\theta<p$. Thus using \eqref{eq:interpolation} in \eqref{eq:besovestimate2}, we find after using Young's inequality,
\begin{align*}
\|v_\e\|_{W^{1,\frac{np}{n-\beta}}(\Omega)}^p\lesssim 1+\frac 1 2\|v_\e\|_{W^{1,\frac{np}{n-\beta}}(\Omega)}^p+C(\theta)\|v_\e\|_{W^{1,p}(\Omega)}^\frac{\theta q}{(\theta q-p)}+\|g\|_{W^{1+\alpha,q}(\Omega)}^q+\|f\|_{L^{q'}(\Omega)}^{q'}.
\end{align*}
Re-arranging and recalling \eqref{eq:pnormlower} we obtain the desired conclusion.

Assume now $p<2$. In this case \eqref{eq:basicBesov} does not hold in this form any more. Instead arguing as before up to this point and using H\"older, we find
\begin{align*}
&\|Dv_\e+Dg-D\tilde v_{\e,h}-D_h \tilde g\|_{L^p(\Omega)}^2 \left(\int_\Omega |Dv_\e|^p+|D\tilde v_{\e,h}|^p+|Dg|^p+|D\tilde g_h|^p\d x\right)^\frac{p-2}{p}\\
\lesssim& 1+\| Dw\|_{L^q(\Omega)}^q+\|g\|_{W^{1+\alpha,q}(\Omega)}^q+\|f\|_{L^{q'}(\Omega)}^{q'}
\end{align*}
In particular, proceeding as before we obtain
\begin{align*}
\|v_\e\|_{B^{1+\frac \alpha 2}_\infty(\Omega)}^2\lesssim& \big(1+\|v_\e\|_{W^{1,q}(\Omega)}^q+\|g\|_{W^{1+\alpha,q}(\Omega)}^q+\|f\|_{L^{q'}(\Omega)}^{q'}\big)\\
&\qquad \times\big(\|v_\e\|_{W^{1,p}(\Omega)}+\|g\|_{W^{1,p}(\Omega)}\big)^\frac{2-p}{p}.
\end{align*}
It is now straightforward to check that applying the interpolation argument from the case $p\geq 2$ gives the desired conclusion.
\end{proof}

\begin{remark}
In the case of autonomous functionals $F(x,z)=F(z)$ the results of this section may be strengthened as follows. Using techniques from interior regularity arguments, the Besov-regularity in tangential directions may be improved when compared to the Besov-regularity obtained in the argument above. Combining this improvement with an embedding of anisotropic Besov spaces into Lebesgue spaces, $\frac{np}{n-\beta}$ may be replaced by a larger value in \eqref{eq:interpolation}. This in turn will imply that larger values of $q$ are allowed in the argument. We intend to return to this observation in further work.
\end{remark}

\section{Regularisation and pointwise minimisers}\label{sec:WB}
 Throughout this section we set $\Theta = 1+n\left(\frac 1 q-\frac 1 p\right)$ if $p< n$, $\Theta = \frac n q$ if $p\geq n$.

The aim of this section is to prove Theorem \ref{thm:nonautonReg} and Proposition \ref{cor:lavrentiev}.
In order to do so we require a number of technical lemmas, which we collect in the following subsection.

\subsection{Technical lemmas}
We begin with an observation already noted in \cite[Theorem 3.1]{Esposito2019}. For completeness we give the proof.
\begin{lemma}\label{lem:changeOfXAlt}
Assume $1<p\leq q\leq \frac{(n+\alpha)p}{n}$. Let $\Omega$ be a domain and $u\in W^{1,p}(\Omega)$. Suppose $F(x,\cdot)$ satisfies \eqref{def:bounds1}-\eqref{def:bounds3} and \eqref{def:changeOfXAlt}. 
Then for $x\in\Omega$ and $\e\leq \min(\e_0,d(x,\p\Omega))$,
\begin{align*}
F(x,Du(\cdot)\star \phi_\e(x))\lesssim 1 + \big(F(\cdot,Du(\cdot))\star \phi_\e\big)(x).
\end{align*}
\end{lemma}
\begin{proof}
Let $C>0$ be an arbitrary fixed constant.
We claim that if $|z|\leq C \e^{-\frac n p}$ and $x\in \overline{B_{\e}(x_0)\cap\Omega}$, then
\begin{align}\label{eq:claim}
F(x,z) \lesssim 1+ \min_{y\in \overline{B_\e(x)\cap \Omega}} F(y,z).
\end{align}
Set $G_\e(x,z)=\min_{y\in \overline{B_\e(x)\cap\Omega}} F(x,z)$.
By \eqref{def:bounds3}, $G_\e(x,z)\geq F(x,z)-\Lambda\e^\alpha (1+|z|^2)^\frac q 2$. Hence for $\delta\in(0,1)$, $\e\in(0,\e_0)$ and $|z|\leq C\e^{-\frac n p}$, using \eqref{eq:pnormlower},
\begin{align*}
G_\e(x,z)\geq &\delta  F(x,z)-\delta\Lambda (1+|z|^2)^\frac q 2+(1-\delta) |z|^p\\
\gtrsim&  \delta F(x,z)-\delta\Lambda C^{q-p}\e^{\alpha+\frac{n(p-q)} p}|z|^p+(1-\delta)|z|^p-\delta(\Lambda\e^\alpha+1)\nonumber\\
\geq& \delta F(x,z)+\big(1-\delta-\Lambda C^{q-p}\delta \e_0^{\alpha+\frac{n(p-q)} p}\big)|z|^p-\Lambda(\e_0^\alpha+1).\nonumber
\end{align*}
The last line relies on $q\leq \frac{(n+\alpha)p}{n}$. Choosing $\delta$ sufficiently small and re-arranging \eqref{eq:claim} follows. 

Note that by Lemma \ref{lem:mollifier}
\begin{align}\label{eq:LinftyBound2}
|Du\star \phi_\e|\leq \|Du\|_{L^p(\Omega)} \e^{-\frac n p}.
\end{align}
Using Jensen's inequality we may estimate for $\e\in(0,\e_0)$ with some fixed $\hat y\in B_\e(x)$,
\begin{align*}
G_\e(x,Du\star \phi_\e)=& F(\hat y, Du\star \phi_\e)\leq \int_{B_1} F(\hat y, Du(y))\phi_\e(x-y)\d y\\
\leq& (F(\cdot,Du(\cdot))\star \phi_\e)(x)
\end{align*}
To obtain the second inequality \eqref{def:changeOfXAlt} was used.
In particular, using \eqref{eq:claim} we conclude
\begin{align*}
F(x,Du\star \phi_\e)\lesssim 1 + \big(F(\cdot,Du(\cdot))\star \phi_\e\big)(x).
\end{align*}
\end{proof}

Partitions of unity feature heavily in our approximation approach. The next lemma studies the behaviour of $F(x,z)$ when $z$ is a sum built using a partition of unity.
\begin{lemma}\label{lem:sumunity}
Let $\Omega$ be a domain with $\Omega=\bigcup_{i\in I} \Omega_i$, where $\Omega_i\subset\Omega$ is open relative to $\Omega$.
Suppose $1<p\leq q<\frac{(n+1)p}{n}$ and suppose $\{\psi_i\,\}_{i\in I}$ is a partition of unity subordinate  to $\{\,\Omega_i\,\}$ and such that 
\begin{align}\label{eq:partitionIntersection}
\|D\psi_i\|_{L^\infty(\Omega_i)}\sim \|D\psi_j\|_{L^\infty(\Omega_j)} \qquad \text{ if }\supp \psi_i\cap \supp \psi_j\neq\emptyset.
\end{align}
 Assume further that there is $M>0$ such that for every $x\in\Omega$, \begin{align}\label{eq:uniformFinite}
{\#\{i\in I\colon x\in\supp \psi_i\}\leq M}.
\end{align}
Take $m'\geq 1$ such that $m'\big(\Theta -\frac{n(q-1)}{p}\left(1-\frac p q\right)\big)\geq 1$.

Suppose $F(x,z)$ satisfies \eqref{def:bounds1} and \eqref{def:bounds2}.
Let $u\in W^{1,p}(\Omega)$. Suppose that there are functions $u_\e^i\in W^{1,p}(\Omega_i)$, $\e_1>0$ and $C=C(\|u\|_{W^{1,p}(\Omega)})$ such that for each $i\in I$, $\e\in(0,\e_1)$,
\begin{enumerate}
\item $u_\e^i\to u$ in $W^{1,p}(\Omega_i)$ as $\e\to 0$ and 
\begin{align}\label{eq:uebounds}
\|u_\e^i-u\|_{L^q(\Omega_i)} \lesssim& C\|D\psi_i\|_{L^\infty(\Omega)}^{-m\Theta}\e^\Theta\\
\|Du_\e^i\|_{L^q(\Omega_i)} \lesssim& C(\e\|D\psi_i\|_{L^\infty(\Omega)}^{-m})^{-\frac n p \left(1-\frac p q\right)}. \nonumber
\end{align}
\end{enumerate}
for some $m\geq m'$.
Then $u_\e= \sum_{i\in I} u_\e^i \psi_i\to u$ in $W^{1,p}(\Omega)$ as $\e\to 0$ and
\begin{align*}
\int_\Omega F(x,Du_\e)\d x\lesssim c\big(\|u\|_{W^{1,p}(\Omega)}\big)+\int_\Omega F\big(x,\sum_{i\in I} Du_\e^i\psi_i\big)\d x.
\end{align*}
\end{lemma}
\begin{proof}
We remark that by the restriction on $q$, $m'$ is well-defined.
Let $\e\in(0,\e_0)$.  We first show that $u_\e\in W^{1,p}(\Omega)$. By Sobolev embedding, \eqref{eq:uniformFinite} and strong convergence of $u_\e^i$, we see that $u_\e\to u$ in  $L^p(\Omega)$. Thus we focus on
\[Du_\e=\sum_{i\in I} Du_\e^i \psi_i+\sum_{i\in I} u_\e^i\otimes D\psi_i=A_1 + A_2.\]
Now due to \eqref{eq:uniformFinite}, $A_1\in L^p(\Omega)$. For $A_2$ note using the fact that $\sum_{i\in I}D\psi_i=0$, Lemma \ref{lem:mollifier} and again \eqref{eq:uniformFinite},
\begin{align*}
\|A_2\|_{L^q(\Omega)} \lesssim C \|\sum_{i\in I}\e^\Theta\|D\psi_i\|_{L^\infty(\Omega)}^{-m\Theta+1} 1_{\supp \psi_i}\|_{L^q(\Omega)}\lesssim C\e^\Theta.
\end{align*}
To obtain the last inequality we also noted the choice of $m'$. Hence $u_\e\to u$ in $W^{1,p}(\Omega)$.
Using \eqref{eq:upperBoundDerivF}, we find using also the bound on $A_2$,
\begin{align*}
\int_\Omega F(x,Du_\e)\d x \lesssim \int_\Omega F(x,A_1)+|A_2|(1+|A_1|^{q-1}+|A_2|^{q-1})\d x\\
\lesssim 1+C+\int_\Omega F(x,A_1)\d x+\int_\Omega |A_1|^{q-1}|A_2|\d x.
\end{align*}
Set for $j\in I$, $I_j=\{i\in I\colon \supp \psi_i\cap \supp \psi_j\neq\emptyset\}$. Then using H\"older's inequality and Lemma \ref{lem:mollifier}
\begin{align*}
\int_\Omega |A_1|^{q-1}|A_2|\d x\leq& \sum_{j\in I} \int_{\Omega_j} |A_2||Du_\e^j|^{q-1}\d x \\
\lesssim& C^q \sum_{j\in I} (\e\|D\psi_j\|_{L^q(\Omega_j)}^{-m})^{-\frac n p\big(1-\frac p q\big)} \sum_{I_j} \e^\Theta \|D\psi_i\|_{L^\infty(\Omega_i)}^{-m\Theta+1} \|1_{\psi_i}\|_{L^q(\Omega)}\\
\lesssim& C^q \e^\tau\sum_{i\in I} \|D\psi_i\|_{L^\infty(\Omega_i)}^{-m\tau+1}\|1_{\psi_i}\|_{L^q(\Omega)}
\end{align*}
with $\tau=\Theta -\frac{n(q-1)}{p}\left(1-\frac p q\right)>0$ by our restriction on $q$. By our choice of $m'$ and \eqref{eq:uniformFinite}, we conclude that the right-hand side is bounded uniformly in $\e$. This concludes the proof.
\end{proof}

We next specialise to a specific choice of $u_\e^i$ and deal with convergence of the term $\int_\Omega F\big(x,\sum_{i\in I} Du_\e^i\psi_i)\big)\d x$.
\begin{lemma}\label{lem:ue}
Let $\Omega$ be a domain. Suppose $1<p\leq q<\frac{(n+\alpha)p}{n}$. Choose $m'\geq 1$ such that $m'\big(\Theta -\frac{n(q-1)}{p}\left(1-\frac p q\right)\big)\geq 1$.
Given $u\in W^{1,p}(\Omega)$, $\e\in\big(0,1/(1+1/6)^n\big)$ and $m\geq 1$ write
\begin{align}\label{def:ue}
 u_{\varepsilon} = \sum_{i\in I} u\star\phi_{\varepsilon \delta_i }\psi_i \qquad\text{ with } \delta_i = |K_i|^\frac{m}{n}.
\end{align}
Here $\{K_i\}_{i\in I}$, a WB-covering of $\Omega$ and $\{\psi_i\}_{i\in I}$ the partition of unity associated to this covering by Theorem \ref{thm:WBunity}.
Assume that $F(x,\cdot)$ satisfies \eqref{def:bounds1}-\eqref{def:bounds3} and \eqref{def:changeOfXAlt}.
Then if $m>m'$, up to passing to a subsequence if necessary, $u_\e\in W^{1,p}_u(\Omega)\cap W^{1,q}_{\tp{loc}}(\Omega)$, $u_\e\to u$ in $W^{1,p}(\Omega)$ and 
\[
\int_\Omega F(x,Du_\e)\d x\to \int_\Omega F(x,Du)\d x \qquad \text{ as } \e\searrow 0.
\]
\end{lemma}
\begin{proof}
The well-definedness of $u_\e$ is ensured by Theorem \ref{thm:covering} and Theorem \ref{thm:WBunity}.

By Theorem \ref{thm:WBunity} and Lemma \ref{lem:mollifier} as well as our choice of $m'$, we may apply Lemma \ref{lem:sumunity} to find that $u_\e\in W^{1,p}(\Omega)$ and $u_\e\to u$ in $W^{1,p}(\Omega)$ as $\e\searrow 0$. By \eqref{def:WBcoverMultiplicity} and Lemma \ref{lem:mollifier} it is clear that $u_\e\in W^{1,q}_{\tp{loc}}(\Omega)$.
We turn to the boundary behaviour of $u_\e$.
By a density argument we may assume that $u\in C(\overline{\Omega})$. Note for $x\in\Omega$,
\begin{align*}
|u_\e(x)-u(x)|\leq \sum_{i\in I: x_n\in K_i} |u\star \rho_{\e\delta_i}(x)-u(x)|\psi_i
\leq M \max_{y\in B_r(x)}  |u(y)-u(x)|
\end{align*}
where $r=\max\{\,\delta_i: i\in I, x\in K_i\,\}$. But if $x\in K_i$, then $\delta_i\lesssim |K_i|^\frac 1 n\sim d(x,\p\Omega)$. It follows that $u_\e\in W^{1,p}_u(\Omega)$.

Finally, again by Lemma \ref{lem:sumunity},
\[\int_\Omega F(x,Du_\e)\d x\lesssim c\big(\|u\|_{W^{1,p}(\Omega)}\big)+\int_\Omega F\big(x,\sum_{i\in I} Du\star \phi_{\e\delta_i}\psi_i\big)\d x.\]

For $\e<\e_0 \left(\max_{i\in I} |K_i|^\frac 1 n\right)^{-1}$, we find using Jensen's inequality, \eqref{def:WBcoverMultiplicity}, Lemma \ref{lem:changeOfXAlt}, Lemma \ref{lem:mollifier} and the dominated convergence theorem,
\begin{align}\label{eq:changeOfXUse}
\int_\Omega F\big(x,\sum_{i\in I} Du\star \phi_{\e\delta_i}\psi_i\big)\leq& \sum_{i\in I}\int_\Omega \int_{B_1} F(x,Du(x-\e\delta_i y))\phi(y)\d y\psi_i\d x\\
\lesssim& 1+\sum_{i\in I} \int_\Omega (F(\cdot,Du(\cdot))\star\phi_{\e\delta_i})(x)\psi_i\d x\nonumber\\
\xrightarrow{\e\to 0}& 1+\sum_{i\in I} \int_\Omega F(x,Du)\psi_i\d x= 1+\F(Du)<\infty.\nonumber
\end{align}
An application of a version of the dominated convergence theorem concludes the proof.
\end{proof}

\subsection{Proof of Theorem \ref{thm:nonautonReg}}
We are now able to prove Theorem \ref{thm:nonautonReg}.
\begin{proof}[Proof of Theorem \ref{thm:nonautonReg}]
Let $u\in W^{1,p}_g(\Omega)$. Let $\Omega^s,\Psi_s,\hat g$ be as constructed in Lemma \ref{lem:diffeos}. Extend $u$ by $\hat g$ to $\Omega^s$, still denoting the extension by $u$. Define 
\begin{align}\label{def:boundaryDiffeoMaps}
F^s(x,z)=F\left(\Psi_s^{-1}(x),z D\Psi_s(\Psi_s^{-1}(x))\right) &\qquad \F^s(u)=\int_\Omega F^s(x,Du)-f\cdot u\d x\\
 u^s(x)&=u(\Psi_s^{-1}(x))\nonumber.
\end{align}

By Lemma \ref{lem:extension}, $u\in W^{1,p}(\Omega^s)$. Since $|D\Psi_s-\tp{Id}|\to 0$ uniformly in $\Omega$, we find ${u^s\to u\in W^{1,p}(\Omega)}$.
With the change of coordinates $\tilde x = \psi^{-1}(x)$ we may compute
\begin{align*}
\int_\Omega F^s(x,Du^s)\d x=&\int_\Omega F(\Psi_s^{-1}(x),Du(\Psi_s^{-1}(x))D\Psi_s^{-1}(x)D\Psi_s(\Psi_s^{-1}(x)))\\
=&\int_\Omega F(\Psi_s^{-1}(x),Du(\Psi_s^{-1}(x))\d x \\
=& \int_{\Omega} F(\tilde x,Du(\tilde x))J_{\Psi_s}\d \tilde x+\int_{\Omega^s\setminus\Omega} F(\tilde x, Du(\tilde x))J_{\psi_s}\d \tilde x \to \int_\Omega F(Du)\d x
\end{align*}
as $s\nearrow 1$, using \eqref{def:bounds2} and absolute continuity of the integral since $u\in W^{1,q}(\Omega^s\setminus\Omega)$.

We check that there is $s_0>0$ such that $F^s(x,z)$ satisfies \eqref{def:bounds1}-\eqref{def:bounds3} and \eqref{def:changeOfXAlt} with constants independent of $s$ for $s\in(0,s_0)$. Using the properties of $\Psi_s$ obtained from Lemma \ref{lem:diffeos} straightforward computations show that \eqref{def:bounds1}-\eqref{def:bounds3} hold if $s\in(0,s_0)$ for some $s_0>0$. Supoose \eqref{def:changeOfXAlt} holds for $F(x,z)$. Let $\e\in(0,\e_0)$. Choosing $\tau$ sufficiently small, we can ensure that $|x-y|\leq \tau\Rightarrow |\psi_s^{-1}(x)-\psi_s^{-1}(y)|<\e$ for all $s\in(0, s_0)$. Thus,
\begin{align*}
\min_{x\in \overline{B_\tau(x_0)\cap\Omega}} F^s(x,z) =& \min_{x\in \overline{B_\e(\Psi_s^{-1}(x_0))\cap\Omega}} F\left(x,z D\Psi_s(\Psi_s^{-1}(x))\right)
\end{align*}
and so $F^s(x,z)$ satisfies \eqref{def:changeOfXAlt}.

Let now $u$ be a pointwise minimiser of $\F(\cdot)$ in the class $W^{1,p}_g(\Omega)$.
Using our observations on $u^s$, Lemma \ref{lem:ue} and a diagonal subsequence argument we find a sequence $u_k\in W^{1,q}_g(\Omega)$ such that $u_k\rightharpoonup u$ weakly in $W^{1,p}(\Omega)$ and in addition
\[\int_\Omega F^{1/k}(x,Du_k)\d x\to \int_\Omega F(x,Du)\d x.\]
 Fix $0\leq \beta<\alpha$. Since $F^s(x,z)$ satisfies \eqref{def:bounds1}-\eqref{def:bounds3} we may apply Lemma \ref{lem:apriori} to find that the pointwise minimisers $v_k$ of 
\begin{align*}
\F^{1/k}_{\e_k}(v) = \int_\Omega F^{1/k}(x,Dv)+\e_k |Dv|^q-f\cdot v\d x
\end{align*}
satisfy for some $\gamma>0$ the apriori estimate
\begin{align*}
\|v_k\|_{W^{1,\frac{np}{n-\beta}}(\Omega)}\lesssim& \left(1+\F^{1/k}_{\e_k}(v_k)+\|g\|_{W^{1+\max\left(\alpha,\frac 1 q\right),q}(\Omega)}+\|f\|_{L^{q'}(\Omega)}\right)^\gamma.
\end{align*}
Here $\e_k = \left(1+\e^{-1}+\|Du_k\|_{L^q(\Omega)}\right)^{-1}$.
Using the minimality of $v_k$ and then the choice of $u_k$ and $\e_k$, we deduce
\begin{align*}
\|v_k\|_{W^{1,\frac{np}{n-\beta}}(\Omega)}\lesssim& \left(1+\F^{1/k}_{\e_k}(u_k)+\|g\|_{W^{1+\max\left(\alpha,\frac 1 q\right),q}(\Omega)}+\|f\|_{L^{q'}(\Omega)}\right)^\gamma\\
\lesssim& \left(1+\F(u)+\|g\|_{W^{1+\max\left(\alpha,\frac 1 q\right),q}(\Omega)}+\|f\|_{L^{q'}(\Omega)}\right)^\gamma.
\end{align*}

In particular we note that the estimate is independent of $s$ and $\e$. Choosing now $\beta$ sufficiently large that we obtain an estimate in $W^{1,q}(\Omega)$, we may extract a subsequence of $v_k$ converging weakly in $W^{1,q}(\Omega)$ to some $v$ and by weak lower semicontinuity of norms, $v\in W^{1,q}(\Omega)$. Moreover by weak semicontinuity and minimality of $v_k$ in the class $W^{1,q}_g(\Omega)$,
\begin{align*}
\F(v) =& \lim_{k\to\infty} \int_\Omega F^{1/k}(x,Dv)-f\cdot v\d x
\leq \lim_{k\to\infty} \F^{1/k}_{\e_k}(x,v_k)\leq \lim_{k\to\infty} \F^{1/k}_{\e_k}(u_k)
= \F(u).
\end{align*}
By minimality of $u$ in the class $W^{1,p}_g(\Omega)$, we conclude $\F(v)=\F(u)$. But then by convexity of $\F(\cdot)$ it follows that $v=u$, concluding the proof.
\end{proof}

We briefly comment on how to prove Corollary \ref{cor:higherDiff}.
\begin{proof}[Proof of Corollary \ref{cor:higherDiff}]
Once Theorem \ref{thm:nonautonReg}  is available, there is no problem in making sense of the Euler-Lagrange equation for $u$. Hence we may repeat the arguments of Lemma \ref{lem:apriori} with $\F_\e(\cdot)$ and $v_\e+g$ replaced by $\F(\cdot)$ and $u$ respectively. The equivalent estimate to \eqref{eq:besovestimate2} and Theorem \ref{thm:embedding} then give the result. In the case of Theorem \ref{thm:regularityRelaxed} using Lemma \ref{lem:relaxedEuler} the same argument applies.
\end{proof}

\subsection{Proof of Proposition \ref{cor:lavrentiev}}
We intend to use Lemma \ref{lem:starshapedLipschitz} to reduce to the case of a Lipschitz domain. Thus we focus first on the case of improving the regularity near the boundary for the sequence $(u_\e)$ from Lemma \ref{lem:ue} in the case of a star-shaped domain.
\begin{lemma}\label{lem:singleDomain}
Let $1<p\leq q<\big(1+\frac{\alpha}{(\alpha+1)n}\big)p$.
Suppose $\Omega$ is a strongly star-shaped Lipschitz domain. Let $\gamma'\Subset\gamma\subset\p\Omega$ where $\gamma'$ and $\gamma$ are connected. Suppose $\Omega'\subset\Omega$ with $\p\Omega'\cap \p\Omega = \gamma'$ and such that if $N_{\gamma'}$ is a neighbourhood of $\gamma'$, then $\Omega'\setminus N_{\gamma'}\Subset\Omega$. If $u\in W^{1,p}(\Omega)\cap W^{1,p}(\p\Omega)\cap W^{1,q}(\gamma)$, then there is $(u_n)\subset W^{1,p}_u(\Omega)\cap W^{1,q}(\Omega')$ such that $u_n\to u$ in $W^{1,p}(\Omega)$ and $\int_{\Omega'} F(x,Du_n)\d x\to \int_{\Omega'} F(x,Du)\d x.$
\end{lemma}
\begin{proof}
Without loss of generality assume that $\Omega$ is strongly star-shaped with respect to $0$.
Let $u\in W^{1,p}(\Omega)\cap W^{1,p}(\p\Omega)\cap W^{1,q}(\gamma)$.
Consider the homogeneous extension of $u$ to $\R^n$ by setting 
\begin{align*}
u(x)=\begin{cases}
	u(x) &\qquad\text{ if } x\in\Omega\\
	1/{\lambda'}u(\lambda' x) &\qquad\text{ if } x\not\in\Omega \text{ where } \lambda' = \sup \{\,\lambda>0: \lambda x\in\overline\Omega\,\}.
	\end{cases}
\end{align*}
For $\frac 1 2<s<x$ we define $u^s(x)=s u(x/s)$. It is easy to check that $u_s\in W^{1,p}_u(\Omega)$ and $u_s\to u$ in $W^{1,p}(\Omega)$ as $s\nearrow 1$. 
Let $\{K_i\}_{i\in I}$ be a WB-covering of $\Omega$ and $\{\psi_i\}_{i\in I}$ the partition of unity associated to this covering by Theorem \ref{thm:WBunity}.
The existence of $\{K_i\}_{i\in I}$ and $\{\psi_i\}_{i\in I}$ is ensured by Theorem \ref{thm:covering} and Theorem \ref{thm:WBunity}.

Take $m\geq 1$ such that $m\big(\Theta -\frac{n(q-1)}{p}\left(1-\frac p q\right)\big)>1$. Define with $\delta_i = |K_i|^\frac m n$,
\begin{align*}
 u_\e^s = \sum_{i\in I} u^s\star\phi_{\varepsilon \delta_i }\psi_i \qquad
 u_\e= \sum_{i\in I} u\star\phi_{\varepsilon \delta_i }\psi_i.
\end{align*}
Note that for $s,\e$ sufficiently small, $u_\e^s$ is well-defined.
By Theorem \ref{thm:WBunity} and Lemma \ref{lem:mollifier}, we may apply Lemma \ref{lem:sumunity} to $u_\e^s$ in order to see that $(u_\e^s)\in W^{1,p}_u(\Omega)$ with $u_\e^s\to u^s$ in $W^{1,p}(\Omega)$ as $\e\to 0$.
 Further by Lemma \ref{lem:mollifier} noting that $u_\e^s$ is $W^{1,q}$-regular in a neighbourhood of $\gamma'$, $u_\e^s\in W^{1,q}(\Omega')$. Using Lemma \ref{lem:sumunity}, 
\begin{align*}
\int_{\Omega'} F(x,Du_\e^s)\d x \lesssim& c(\|u^s\|_{W^{1,p}(\Omega)})+ \int_{\Omega'} F\big(x,\sum_{i\in I} Du_\e^s\psi_i\big)\d x\\
\lesssim& c\big(\|u\|_{W^{1,p}(\Omega)}+\|u\|_{W^{1,p}(\p\Omega)}\big)+\int_{\Omega'}F\big(x,\sum_{i\in I} Du^s\star\phi_{\e\delta_i}\psi_i\big)\d x. 
\end{align*}

 Let $1-s$ be sufficiently small that $\p\big(s^{-1}\Omega'\setminus\Omega\big)\subset\gamma$ and in addition for any neighbourhood $N_{\gamma}$ of $\gamma$,  $(s^{-1}\Omega'\cap\Omega)\setminus N_\gamma\Subset\Omega$. Then using the change of coordinates $x\to s x$,
\begin{align*}
\int_{\Omega'} F(x,\sum_{i\in I} Du_\e^s\psi_i)\d x=& s^{-n}\int_{s^{-1}\Omega'} F(s x,A)\d x\\
\leq & s^{-n} \int_{s^{-1}\Omega'\setminus \Omega}F(s x,A)\d x+s^{-n}\int_{s^{-1}\Omega'\cap\Omega} F(s x,A)\d x
\end{align*}
with $A= \sum_{i\in I} \big(Du(\cdot)\star \phi_{\e\delta_i}\big)(x)\psi_i(s x)$. Now note that $\|u\|_{ W^{1,q}(s^{-1}\Omega'\setminus\Omega)}\lesssim \|u\|_{W^{1,q}(\gamma)}$, so using \eqref{def:bounds2}, Lemma \ref{lem:mollifier} and \eqref{def:WBcoverMultiplicity}, we see that
\begin{align*}
s^{-n} \int_{s^{-1}\Omega'\setminus \Omega}F(s x,A)\d x\lesssim& s^{-n}\int_{s^{-1}
\Omega'\setminus\Omega} 1+|A|^q\d x
\lesssim 1+\|u\|_{W^{1,q}(\gamma)}
\end{align*}
independently of $s$ and $\e$.
Note that if $x\in\supp \psi_i(s\cdot)\cap \Omega$, then $|K_i|^\frac 1 n\sim d(x,\p\Omega)\gtrsim (1-s)$ and so $\delta_i\gtrsim (1-s)^m$. Thus we find for some sufficiently small $c>0$ and using Lemma \ref{lem:mollifier} as well as \eqref{def:bounds3},
\begin{align*}
\int_{s^{-1}\Omega'\cap\Omega} (1-s)^\alpha|A_1|^q\d x\lesssim (1-s)^\alpha\int_{s^{-1}\Omega'\cap\Omega} \sum_{i\in I: |K_i|^\frac 1 n\geq c(1-s)}(\e\delta_i)^{-\frac n p(q-p)}\|u\|_{W^{1,p}(K_i)}^p.
\end{align*}
Recalling the definition of $\delta_i$, if $\alpha-m\frac n p(q-p)\geq 0$, then the right-hand side is bounded independently of $s$.
Recalling the choice of $m$ we find that we need to have ${q<(1+\frac{\alpha}{(\alpha+1)n}) p}$.

By a version of the dominated convergence theorem we conclude since $Du_\e^s\to Du^s$ almost everywhere in $\Omega'$,
\begin{align*}
\int_{\Omega'} F(x,Du_\e^s)\d x\to \int_{\Omega'} F(x,Du_\e)\d x \qquad \text{as }s\nearrow 1.
\end{align*}
By Lemma \ref{lem:ue}
\begin{align*}
\int_{\Omega'} F(x,Du_\e)\d x\to \int_{\Omega'} F(x,Du)\d x \qquad \text{ as } \e\searrow 0.
\end{align*}
Thus by a diagonal subsequence argument, we can extract a subsequence of $(u_\e^s)$ with all the desired properties.
\end{proof}

We can now finally prove Proposition \ref{cor:lavrentiev}.
\begin{proof}[Proof of Proposition \ref{cor:lavrentiev}]
Note that the first part of Proposition \ref{cor:lavrentiev} is a consequence of Lemma \ref{lem:ue} and Lemma \ref{lem:lavrientiev}.

For the second part
by Lemma \ref{lem:lavrientiev} it suffices to exhibit a sequence $(u_n)\subset W^{1,q}_g(\Omega)$ such that $u_n\rightharpoonup u$ weakly in $W^{1,p}(\Omega)$ and $\F(u_n)\to \F(u)$ as $n\to\infty$. By Lemma \ref{lem:starshapedLipschitz} we may write
\[\Omega =\bigcup_{i=1}^N \Omega_i = \bigcup_{i=1}^N \omega_i\]
where $\Omega_i$ is a strongly star-shaped Lipschitz domain, $\omega_i\Subset\Omega_i$ relative to $\Omega$ and further ${\p\omega_i\cap\p\Omega \Subset \p\Omega_i\cap\p\Omega}$ if $\p\Omega_i\cap\Omega\neq\emptyset$. We may also assume that $u\in W^{1,p}(\p\Omega_i)$ for $i=1,\ldots,N$ with $\|u\|_{W^{1,p}(\p\Omega_i)}\lesssim \|u\|_{W^{1,p}(\Omega_i)}+\|g\|_{W^{1+1/q,q}(\Omega)}$. Let $\psi_i$ be a partition of unity subordinate to $\omega_i$. By Lemma \ref{lem:singleDomain}, there are sequences $(u_n^i)\subset W^{1,p}(\Omega_i)\cap W^{1,q}(\omega_i)$ such that $u_n^i\to u$ in $W^{1,p}(\Omega)$ and $\int_{\omega_i} F(x,Du_n^i)\d x\to \int_{\omega_i} F(x,Du)\d x$ where $i=1,\ldots,N$. Consider
\begin{align*}
u_n = \sum_{i=1}^N u_n^i\psi_i.
\end{align*}
Then $u_n\in W^{1,q}_g(\Omega)$ and $u_n\to u$ in $W^{1,p}(\Omega)$. Moreover using \eqref{eq:upperBoundDerivF}, H\"older's inequality, Sobolev's embedding and
Jensen's inequality,
\begin{align*}\int_\Omega F(x,Du_n)\d x\lesssim& \sum_{i,j=1}^N\int_\Omega |u_n^i\otimes D\psi_i|(1+|u_n^i\otimes D\psi_j|^{q-1}+|Du_n^j\otimes D\psi_j|^{q-1})\d x\\
&\qquad+\int_\Omega F\big(x,\sum_{i=1}^N Du_n^i\psi_i\big)\d x\\
\lesssim& \sum_{i,j=1}^n \|Du_n^i\|_{L^p(\Omega_i)}^{q-1}\|Du_n^j\|^q_{L^p(\Omega_j)}+\sum_{i=1}^N\int_{\omega_i} F(x,Du_n^i)\psi_i\d x\\
\to& \sum_{i,j=1}^n \|Du_n^i\|_{L^p(\Omega_i)}^{q-1}\|Du_n^j\|_{L^p(\Omega_j)}^q+\int_{\Omega} F(x,Du)\d x<\infty.
\end{align*}
To obtain the second line we used the assumption on $q$.
Thus we conclude by a variant of the domianted convergence theorem that ${\F(u_n)\to \F(u)}$ as $n\to\infty$.
\end{proof}

\section{Regularity under alternatives to assumption \texorpdfstring{\eqref{def:bounds1}}{}}\label{sec:alternatives}
In this section we explore some alternatives to \eqref{def:bounds1} which only require minor modifications of the proofs. To be precise we consider the following assumptions:
Assume one of the following holds for some $\mu\geq 0, \lambda>0$ and for all $z,w\in \R^{n\times m}$ and almost every $x\in\Omega$: 
\begin{align}
\label{def:px}
&\begin{cases}
 \lambda \left(\mu^2+|z|^2+|w|^2\right)^\frac{p(x)-2}{2}|z-w|^2\leq F(x,z)-F(z,w)-\p_z F(x,w)\cdot (z-w) \\[10pt]
\text{ where } p(x)\in C^{0,\alpha}(\Omega) \text{ and there is } \e>0 \text{ such that } 1<p\leq p(x)\leq q 
\end{cases}\tag{H1.1}\\
\label{def:anisotropic}
&\begin{cases}
 \sum_{i=1}^n \left(\mu_i^2+|z_i|^2+|w_i|^2\right)^\frac{p_i-2}{2}|z_i-w_i|^2\leq F(x,z)-F(z,w)-\p_z F(x,w)\cdot (z-w) \\[10pt]
\text{ where } z_i = (z_i)^j_{1\leq j\leq n}, w_i = (w_i)^j_{1\leq j\leq n} \text{ and } 1< p =p_1 \leq p_2\leq...\leq p_n = q
\end{cases}\tag{H1.2}
\end{align}
Before commencing to prove versions of Theorem \ref{thm:regularityRelaxed} and Theorem \ref{thm:nonautonReg}, we remark that replacing \eqref{def:bounds1} with \eqref{def:px} or \eqref{def:anisotropic} also causes a change in the proof of Lemma \ref{lem:convApproximate}: Using the notation of Lemma \ref{lem:convApproximate}, \eqref{eq:convexArg} needs to be replaced by
\begin{align*}
\overline \F\left(\frac{w_1+w_2}{2}\right)+\frac{\nu}{p}\left(\int_\Omega |Dw_1-Dw_2|^{p(x)}\d x\right)^\frac 1 p\leq \frac 1 2 \left(\overline \F(w_1)+\overline \F(w_2)\right).
\end{align*}
or
\begin{align*}
\overline \F\left(\frac{w_1+w_2}{2}\right)+\frac{\nu}{p}\left(\int_\Omega \sum_{i=1}^n |D_i w_1-D_i w_2|^{p_i}\d x\right)^\frac 1 p\leq \frac 1 2 \left(\overline \F(w_1)+\overline \F(w_2)\right),
\end{align*}
respectively. Arguing as before, either estimate suffices to conclude that $Du_\e \to Du$ in $L^p(\Omega)$.

We then have:
\begin{theorem}\label{thm:modified}
Suppose $g\in W^{1+\alpha,q}(\Omega)$ and $F(x,z)$ satisfies assumptions \eqref{def:bounds2}, \eqref{def:bounds3} and either \eqref{def:px} or \eqref{def:anisotropic}. Suppose $1\leq p<q<\frac{(n+\alpha)p}{n}$. Then the conclusion of 
Theorem \ref{thm:regularityRelaxed} still holds. If in fact $g\in W^{1+\max(\alpha,1/q),q}(\Omega)$ and $F(x,\cdot)$ satisfies in addition  \eqref{def:changeOfXAlt}, then the conclusion of Theorem \ref{thm:nonautonReg} holds.
\end{theorem}
\begin{proof}
For simplicity we only focus on the case $p\geq 2$. The case $p<2$ follows from similar considerations.
Assume first that \eqref{def:px} holds. Using the notation of Section \ref{sec:lipschitzDomains}, extend $p(x)$ to a $C^{0,\alpha}$-function on $B(0,R)$ with $p(x)\in[p,q]$ by setting 
\[
p(x)=\inf_{y\in\Omega} \big(p(y)+\|p\|_{C^{0,\alpha}(\Omega)}|x-y|\big).
\]
 for $x\in B(0,R)\setminus\Omega$. 
The first change to the argument occurs in the proof of the a-priori estimate  in Lemma \eqref{lem:apriori}. Instead of \eqref{eq:lowerEstimate} we obtain:
\begin{align}
\F_\e(v+g)-\F_\e(v_\e+g)
\gtrsim&\int_\Omega \left(\mu^2+|Dv+Dg|^2+|Dv_\e+Dg|^2\right)^\frac{p(x)-2}{2}|Dv-Dv_\e|^2\d x\nonumber,
\end{align}
We now argue as before to conclude,
\begin{align*}
\int_\Omega |Dv_\e|^\frac{np(x)}{n-\beta}\d x\lesssim \left(1+\frac 1 2\|v_\e\|_{W^{1,\frac{np}{n-\beta}}(\Omega)}^p+C(\theta)\|v_\e\|_{W^{1,p}(\Omega)}^\frac{\theta q}{(\theta q-p)}+\|g\|_{W^{1+\alpha,q}(\Omega)}^q+\|f\|_{L^{q'}(\Omega)}^{q'}\right).
\end{align*}
Noting that 
\begin{align*}
\int_\Omega |Dv_\e|^\frac{np}{n-\beta}\d x\lesssim \int_\Omega |Dv_\e|^\frac{np(x)}{n-\beta}\d x + 1.
\end{align*}
we finish the proof without further change.

For the proof of Theorem \ref{thm:nonautonReg} we only need to verify that \eqref{def:px} holds for $F^s(x,\cdot)$. The computation is straightforward.

We now assume that \eqref{def:anisotropic} holds. Instead of \eqref{eq:lowerEstimate} we have
\begin{align*}
\F_\e(v+g)-\F_\e(v_\e+g)\gtrsim \int_\Omega \sum_{i=1}^n\left(\mu_i^2+|Dv_i+Dg_i|^2+|Dv_{\e,i}+Dg_i|^2\right)^\frac{p_i-2}{2}|Dv_i-Dv_{\e,i}|^2\d x\nonumber.
\end{align*}
Repeating the arguments given in Lemma \ref{lem:apriori} we use this to obtain
\begin{align*}
\int_\Omega \sum_{i=1}^n |Dv_{\e,i}|^\frac{np_i}{n-\beta}\d x\lesssim \left(1+\frac 1 2\|v_\e\|_{W^{1,\frac{np}{n-\beta}}(\Omega)}^p+C(\theta)\|v_\e\|_{W^{1,p}(\Omega)}^\frac{\theta q}{(\theta q-p)}+\|g\|_{W^{1+\alpha,q}(\Omega)}^q+\|f\|_{L^{q'}(\Omega)}^{q'}\right).
\end{align*}
The proof now concludes by the same arguments as before by noting that
\begin{align*}
\int_\Omega |Dv_\e|^\frac{np}{n-\beta}\d x\lesssim 1+\int_\Omega \sum_{i=1}^n |Dv_{\e,i}|^\frac{n p_i(x)}{n-\beta}\d x
\end{align*}
In the proof of Theorem \ref{thm:nonautonReg} we need to verify that \eqref{def:anisotropic} holds for $F^s(x,\cdot)$. The computation is straightforward.
\end{proof}

\begin{remark}
It is straightforward to adapt the arguments of Theorem \ref{thm:modified} to to growth conditions that combine \eqref{def:px} and \eqref{def:anisotropic}:
\begin{align}\label{def:combined}
\begin{cases}
\sum_{i=1}^n \left(\mu_i^2+|z_i|^2+|w_i|^2\right)^\frac{p_i(x)-2}{2}|z_i-w_i|^2\leq F(x,z)-F(z,w)-\p_z F(x,w)\cdot (z-w) \\[10pt]
\text{ where } z_i = (z_i)^j_{1\leq j\leq n}, w_i = (w_i)^j_{1\leq j\leq n} \text{ and } 1< p\leq p_i(x)\leq q.
\end{cases}\tag{H1.3}
\end{align}
\end{remark}

\section{Examples}\label{sec:examples}
In this section we list a number of examples to which our theory applies. We highlight in particular that we can treat the double-phase functional \ref{eq:2}, the anisotropic $p(x)$-Laplacian \ref{eq:4} as well as more general anisotropic functionals, e.g. \ref{eq:3}.
The theory developed in this paper applies to all the functionals listed below:\begin{enumerate}[itemsep=6pt]
\item $\F_1(u)=\int_\Omega a(x)F(Du)\d x$, where $1\leq a(\cdot)\leq L$,
\item $\F_2(u)=\int_\Omega \sum_{i=1}^n a_i(x) F_i(D_i u)\d x$, where $1\leq a_i(\cdot)\leq L$,
\item \label{eq:2}$\F_3(u)=\int_\Omega |Du|^p+a(x)|Du|^q\d x$ where $0\leq a(x)\in C^{0,\alpha}(\Omega)$,
\item\label{eq:3} $\F_4(u)=\int_\Omega |Du|^p+|a^{i,j}_{\alpha,\beta}(x) D_i u^\alpha D_j u^\beta|^\frac q 2\d x$ where $a^{i,j}_{\alpha,\beta}(\cdot)\in C^{0,\alpha}(\Omega)$ and for all $x\in \Omega$ and $\xi\in \R^{n\times m}$,
\begin{align*}
\lambda(x)|\xi|^2 \leq a_{\alpha,\beta}^{i,j}\xi_i^\alpha \xi_j^\beta \leq L\lambda(x)|\xi|^2
\end{align*}
with $L\geq 1$ and $0\leq \lambda(\cdot)\in C^{0,\alpha}(\Omega)$.
where $0\leq a_i(\cdot)\in C^{0,\alpha}(\Omega)$, and further $1<p\leq p_i\leq q$,
\item\label{eq:4} $\F_5(u)=\int_\Omega \sum_{i=1}^n |D_i u|^{p_i(x)}\d x$, where $p\leq p_i(x)\leq q$ and $p_i(x)\in C^{0,\alpha}(\Omega)$,
\item\label{eq:6} $\F_6(u)=\int_\Omega |Du|^{p(x)} \log(1+|Du|)\d x$, where $1<p\leq p(x)\leq q$,
\item $\F_7(u)=\int_\Omega |Du|^q+a(x)\max(|D_n u|,0)\d x$ where $q>2$, $0\leq a(\cdot)\in C^{0,\alpha}(\Omega)$,
\item\label{eq:7} $\F_8(u)=\int_\Omega F(x,Du)\d x$ where $F(x,z)=h(a(x),z)$ where 
	\begin{enumerate}
	\item $t\to h(t,z)$ is increasing
	\item $h(x,z)$ is convex in the second argument
	\item $a(x)\in C(\overline\Omega)$
	\item $F(x,z)$ satisfies \eqref{def:bounds1}-\eqref{def:bounds3} (or one of \eqref{def:px}, \eqref{def:anisotropic} instead of \eqref{def:bounds1})	
	\end{enumerate}
\end{enumerate}
The first five examples are standard. \ref{eq:6} has been studied as a model case (with $p(x)=\const>1$) in \cite{Marcellini1991}. \ref{eq:7} is inspired by \cite{Qi1993}, whereas the last example is taken from \cite{Esposito2019}.

Before making two observations that are helpful to verify that $W^{1,q}$-regularity holds for minimisers of these examples, let us extract the precise statement for the proof of which we use \eqref{def:changeOfXAlt}.
\begin{corollary}\label{cor:xdependSummary}
Suppose $\Omega$ is a $C^{1,\alpha}$-domain.
Suppose $g\in W^{1+\max\left(\alpha,\frac 1 q\right),q}(\Omega)$. Assume $\F(\cdot)$ satisfies \eqref{def:bounds1}-\eqref{def:bounds3} with $1< p\leq q<\frac{(n+\alpha)p}{n}$. Suppose $u$ is a pointwise minimiser of $\F(\cdot)$ in the class $\in W^{1,p}_g(\Omega)$ and one of the following holds:
\begin{enumerate}
\item\label{eq:weakqseq} There is a sequence $u_k\in W^{1,p}_g(\Omega)\cap W^{1,q}_{\tp{loc}}(\Omega)$ such that $u_k\rightharpoonup u$ weakly in $W^{1,p}(\Omega)$ and $\F(u_k)\to \F(u)$ as $k\to\infty$.
\item\label{eq:commuteMollifier} There is $\e_0>0$ such that for $\e\in(0,\e_0)$,
	\begin{align*}
	F(x,Du\star \phi_{\e})\lesssim 1 + (F(\cdot,Du(\cdot))\star \phi_{\e})(x).
	  \end{align*}
\end{enumerate}
 Then $u\in W^{1,q}(\Omega)$.
\end{corollary}
\begin{proof}
Item \ref{eq:weakqseq} is precisely the conclusion from Lemma \ref{lem:lavrientiev} we need in order to run the proof of Theorem \ref{thm:nonautonReg}. Item \ref{eq:commuteMollifier} is the conclusion of Lemma \ref{lem:changeOfXAlt} which is the only place \eqref{def:changeOfXAlt} is used. Hence the proof of Theorem \ref{thm:nonautonReg} remains unchanged.
\end{proof}

We highlight two ways of combining functionals for which the assumptions of Corollary \ref{cor:xdependSummary} hold. First, if each $F_i(x,z)$, $i\in\{\,1,\ldots,N\,\}$ for some $N\in\N$ satisfies assumption \ref{eq:commuteMollifier} of Corollary \ref{cor:xdependSummary}, then $\sum_i F_i(x,z)$ also satisfies the assumption. 

Second, assume $G(x,z)$ satisfies\eqref{def:bounds1}-\eqref{def:bounds3} and \ref{eq:weakqseq} in Corollary \ref{cor:xdependSummary}. Consider $F(x,z)$ and suppose that
\begin{align}\label{eq:asymptoticBounds}
G(x,z)-1\lesssim F(x,z)\lesssim G(x,z)+1.
\end{align}
Then from Lemma \ref{lem:lavrientiev} applied to $G(x,z)$ there is $u_\e\in W^{1,p}(\Omega)\cap W^{1,q}_{\tp{loc}}(\Omega)$ such that $u_\e\rightharpoonup u$ weakly in $W^{1,p}(\Omega)$ and $\int_\Omega G(x,Du_\e)\d x\to \int_\Omega G(x,Du)\d x$. Thus by a version of the dominated convergence theorem and \eqref{eq:asymptoticBounds} ${\int_\Omega F(x,Du_\e)\d x\to \int_\Omega F(x,Du)\d x}$. Hence $F(x,z)$ also satisfies assumption \ref{eq:weakqseq} of Corollary \ref{cor:xdependSummary}.

\bibliographystyle{plain}

\end{document}